\documentclass[11pt]{amsart}
\usepackage{url,enumitem,amsopn,cite}

\newcommand{\cal}{\mathcal}

\newcommand{\comps}{\mbox{$\mathbb C$}}
\newcommand{\rats}{\mbox{$\mathbb Q$}}
\newcommand{\nats}{\mbox{$\mathbb N$}}
\newcommand{\ints}{\mbox{$\mathbb Z$}}
\newcommand{\power}{\mbox{$\mathbb P$}}

\newcommand{\supp}{{\mathop{\mathrm{supp}}\nolimits}}

\newcommand{\lcm}{{\mathop{\mathrm{lcm}}\nolimits}}
\newcommand{\comment}[1]{}

\usepackage{amsmath, amsfonts, amssymb, latexsym, stmaryrd}
\usepackage{graphicx}                   

\def\squarebox#1{\hbox to #1{\hfill\vbox to #1{\vfill}}}
\def\qed{\hspace*{\fill}
        \vbox{\hrule\hbox{\vrule\squarebox{.667em}\vrule}\hrule}\smallskip}

\theoremstyle{plain}
\newtheorem{lemma}{Lemma}[section]
\newtheorem{theorem}[lemma]{Theorem}
\newtheorem{corollary}[lemma]{Corollary}
\newtheorem{proposition}[lemma]{Proposition}
\theoremstyle{definition}
\newtheorem{claim}[lemma]{Claim}
\newtheorem{observation}[lemma]{Observation}
\newtheorem{definition}[lemma]{Definition}

\newtheorem{example}{Example}
\newtheorem*{rmk*}{Remark}
\newtheorem*{rmks*}{Remarks}
\newtheorem*{conventions*}{Conventions}
\newtheorem*{convention*}{Convention}

\def\squareforqed{\hbox{\rlap{$\sqcap$}$\sqcup$}}
\def\qed{\ifmmode\squareforqed\else{\unskip\nobreak\hfil
\penalty50\hskip1em\null\nobreak\hfil\squareforqed
\parfillskip=0pt\finalhyphendemerits=0\endgraf}\fi}

\newlength{\tablength}
\newlength{\spacelength}
\newcommand{\tabstar}{\hspace*{\tablength}}
\newcommand{\spacestar}{\hspace*{\spacelength}}
\def\obeytabs{\catcode`\^^I=\active}
{\obeytabs\global\let^^I=\tabstar}
{\obeyspaces\global\let =\spacestar}
\newenvironment{display}{\begingroup\obeylines\obeyspaces\obeytabs}{\endgroup}
\newenvironment{prog}{\begin{display}\parskip0pt\sf}{\end{display}}

\author{Geir Agnarsson}
\address{Department of Mathematical Sciences \\
  George Mason University \\ Fairfax, VA  22030}
\email{geir@math.gmu.edu}

\author{Jim Lawrence}
\address{Department of Mathematical Sciences \\
  George Mason University \\ Fairfax, VA  22030}
\email{lawrence@gmu.edu}

\title{Power-closed ideals of\\ polynomial and Laurent polynomial rings}

\date{\today}

\begin{document}

\begin{abstract}
We investigate the structure of power-closed ideals of the
complex polynomial ring $R = \comps[x_1,\ldots,x_d]$
and the Laurent polynomial ring 
$R^{\pm} = \comps[x_1,\ldots,x_d]^{\pm} = M^{-1}\comps[x_1,\ldots,x_d]$, where
$M$ is the multiplicative sub-monoid $M = [x_1,\ldots,x_d]$ of $R$.
Here, an ideal $I$ is {\em power-closed}
if $f(x_1,\ldots,x_d)\in I$ implies $f(x_1^i,\ldots,x_d^i)\in I$
for each natural $i$. In particular, we investigate related closure
and interior operators on the set of ideals of $R$ and $R^{\pm}$.
Finally, we give a complete description of principal power-closed
ideals and of the radicals of general power-closed
ideals of $R$ and $R^{\pm}$.
\end{abstract}
\maketitle

\section{Introduction}
\label{sec:intro}

The original motivation for the questions studied in this paper came
largely from
the theory of convex sets and in particular from results concerning
Minkowski addition of convex sets.
Let ${\cal S}$ denote the additive group of functions on a Euclidean
space that is generated by the indicator functions $[P]$ of convex
polytopes $P$.
If $P$ and $Q$ are nonempty convex polytopes then their
{\it Minkowski sum} is the set $P+Q = \{p+q : p\in P, q\in Q\}$,
which is also a convex polytope.
A well-defined product, $\cdot$, on ${\cal S}$ is uniquely determined by the
condition that, for nonempty convex polytopes $P$ and $Q$,
$[P]\cdot[Q] = [P+Q]$.
With this product, ${\cal S}$ is a commutative ring.
Its multiplicative identity is the indicator function of the singleton set
$\{0\}$.
The ring structure on ${\cal S}$ has been considered by several authors,
e.~g., \cite{Groemer, McMullen, Lawrence, Jay-Klaus, Morelli}.

In~\cite{Lawrence, Jay-Klaus},
${\cal S}$  and its subrings that are generated by
indicator functions of convex polytopes are called {\it Minkowski rings}.
(The paper~\cite{Lawrence} is unpublished; however, preprints were
rather widely
disseminated, and~\cite{Jay-Klaus} was based upon it.)
These rings possess a rather striking property.
For each $k\in \nats$ and for each convex polytope $P$, the power $[P]^k$
in the Minkowski ring is simply $[kP]$, the indicator function of
the dilation of $P$ by the factor $k$.
The function that takes $[P]$ to $[kP]$ extends to a homomorphism,
$\psi_k:{\cal S}\rightarrow {\cal S}$.  Taking that as a cue, we consider
rings $S$, commutative rings with identity, that mimic the Minkowski
rings in the following way: The ring $S$ possesses a distinguished
subset $G$ that generates $S$ as a ring, together with operators
$\psi_k:S\rightarrow S$ for $k\in \nats$  such that each operator is
a homomorphism of $S$ to itself and such that, for each $g\in G$,
$\psi_k(g) = g^k$. (Notice that $G$ can be taken to be closed
under multiplication if so desired.)
We dub these {\it Adams-Minkowski rings}.

A reader familiar with them may have already recognized that these rings
are related to $\lambda$-rings, a notion that
was originally motivated by rings of vector bundles and operations on them,
i.e., $K$-theory.

Terminology
applying to $\lambda$-rings varies; we use the second section
of~\cite{Atiyah-Tall}. (For much more on the topic of $\lambda$-rings,
 see the book,~\cite{Yau}.)
Special $\lambda$-rings which are generated by elements of dimension
at most 1 are Adams-Minkowski rings, where the set of elements of degree at
most 1 can be taken as the distinguished generating
set and the operators $\psi_k$ are the Adams operations (whose presence
prompted the ``Adams'' in the name).
(Special $\lambda$-rings are sometimes called $\lambda$-rings, in which
case, usually, $\lambda$-rings are called pre-$\lambda$-rings.)
Conversely, Adams-Minkowski rings in which division by natural numbers
is possible are special $\lambda$-rings which are generated by elements
of dimension at most 1 and for which the Adams operations are the
homomorphisms $\psi_k$.
As shown in~\cite{Atiyah-Tall}, any special
$\lambda$-ring may be obtained as a subring of
such a special $\lambda$-ring,
closed under the operations; indeed the special $\lambda$-rings may be
characterized as such subrings.

Although many results here will apply to Adams-Minkowski rings in
general, or to large subclasses,
the scope in this paper will be considerably reduced
in order to accommodate its most acute results.
First, attention is restricted to Noetherian Adams-Minkowski rings
$S$.  The polynomial ring  $\ints[x_1, \ldots, x_d]$ is such a ring, with $G$
taken to be $G = \{1, x_1,\ldots,x_d\}$;
and, of course, all other Noetherian Adams-Minkowski rings
can be obtained as homomorphic images of such
polynomial rings.  We are interested in the ideals that are the kernels of
certain of these homomorphisms, and it will be useful to consider
their analogues
in the ring $R = \comps[x_1,\ldots,x_d]$ of polynomials with complex
coefficients, instead.  Also, we will address their analogues in the
rings of Laurent polynomials having complex coefficients.

\section{Definitions and observations}
\label{sec:defs-obs}

Although some results will be valid in Adams-Minkowski
rings more generally, we will consider
the polynomial ring $R = \comps[x_1,\ldots,x_d]$ over the complex
field $\comps$ in $d$ variables.

For a polynomial $f = f(\tilde{x}) = f(x_1,\ldots,x_d)\in R$ let 
$f^{(i)}(\tilde{x}) = f(x_1^i,\ldots,x_d^i)$.
\begin{definition}
\label{def:closed}
An ideal $I$ of $R = \comps[x_1,\ldots,x_d]$ is 
{\em power-closed} if $f\in I$ implies
$f^{(i)}\in I$ for each $i\in\nats$.
\end{definition}

Notice that $R$ itself is power-closed, and the intersection of an
arbitrary family of power-closed ideals is power-closed.  Therefore
the function that takes an arbitrary subset $S$ of $R$ to the 
smallest power-closed ideal containing it (i.e., the intersection
of the collection of power-closed ideals containing it) is a
closure operator on $R$: It is enlarging, monotone, and idempotent.
For a given $S$,
this ideal will be denoted by $S^{(*)}$.  It follows that the collection
of power-closed
ideals forms a complete lattice in which the meet operation is intersection.
This is of course also the meet operation in the lattice of all ideals of
$R$, and it will be seen shortly that the join operation in the lattice
of power-closed ideals is also the join operation (namely, sum) in the
lattice of ideals.  It will follow
that the lattice of power-closed ideals is a complete sublattice of
the lattice of ideals.

Many types of ideals are power-closed ideals, as the following examples show.
(i) Each monomial ideal generated by a finite set of monomials
from the monoid $[x_1,\ldots,x_d]$ is a power-closed ideal.
(ii)  Each ideal of $R$ that is generated by a finite set of binomials
each of which is a difference of two monomials is a power-closed ideal.  
(iii) In particular each toric ideal,
being a prime ideal generated by binomials, is a power-closed ideal.
(iv) More generally, every ideal generated by polynomials each being
a product of monomials and binomials from 
$R = \comps[x_1,\ldots,x_d]$ is a power-closed ideal. 
We will see in the next section that there are power-closed ideals of 
$R$ that are not of the most general type (iv) mentioned here above.

For a subset $S\subseteq R$ we can also consider the ideal
$S^{(*)}$ of $R$ constructed ``from below''. 
For a subset $S\subseteq R$ and $i\geq 1$ let
$S^{(i)} : = ( f^{(\ell)} : f\in S, 1\leq\ell\leq i)$.
In this way we have a nested family of ideals containing $S$; starting
with the set $S$:
\[
S\subseteq S^{(1)}\subseteq S^{(2)}\subseteq\cdots\subseteq
S^{(i)}\subseteq\cdots.
\]
The union is an ideal $S^{(\star)} = \bigcup_{i\geq 0} S^{(i)}$
of $R$ that contains $S$.
An element $g$ of $R$ is in $S^{(\star)}$ if and only if
$g = h_1f_1^{(\ell_1)} + \cdots + h_kf_k^{(\ell_k)}$
for some $h_1,\ldots,h_k\in R$, $f_1,\ldots,f_k\in S$ and positive integers
$\ell_1,\ldots,\ell_k$. For each $i\geq 1$ we have
$g^{(i)} = h_1^{(i)}f_1^{(i\ell_1)} + \cdots
+ h_k^{(i)}f_k^{(i\ell_k)}\in S^{(\star)}$;
it follows that  $S^{(\star)}$ is a power-closed ideal. Clearly, any
power-closed ideal that contains $S$ must contain $S^{(\star)}$ and
hence $S^{(\star)} = S^{(*)}$ from above.

Since $R$ is Noetherian, every ideal $I = (S)$ is finitely generated,
so $S$ contains a finite subset $S'\subseteq S$ with $I = (S')$.
Consequentially, there is an $N\in\nats$ with $I^{(*)} = I^{(N)}$.
In a similar vein as the above, we can show that $I^{(*)}$ can be written
in terms of any set of generators for $I$. We have the following.
\begin{claim}
\label{clm:*-gen}
For an ideal $I$ of $R$ we have that $I^{(*)}$ is a power-closed ideal
and contains
$I$. Further, if $I = (f_1,\ldots,f_n)\in R$ then for some $N\in\nats$
\[
I^{(*)} = I^{(N)} = (
f_1,f_1^{(2)},\ldots,f_1^{(N)},f_2,f_2^{(2)},\ldots,f_2^{(N)}, \ldots,
f_n,f_n^{(2)},\ldots,f_n^{(N)}).
\]
\end{claim}
As a consequence, for any $f_1, \ldots, f_n\in R$, the ideal
$\sum_i(f_i)^{(*)}$ is power-closed, and every power-closed ideal is of this
form.   The sum of any collection of power-closed ideals is again
power-closed.  Notice that this implies that the join operation in the lattice
of power-closed ideals is sum.

From the above discussion and Claim~\ref{clm:*-gen}
we see that for any ideal $I$ we can obtain the {\em power-closure}
$I^{(*)}$ of $I$ directly
from the generators of $I$.
As mentioned earlier in this section $S\mapsto S^{(*)}$ is a closure
operator on subsets of $R$ and therefore also on the set of ideals
of $R$. That $I\mapsto I^{(*)}$ is a closure operator for ideals $I$ of $R$
follows from the power-closedness of $I^{(*)}$; that is,
$I^{(**)} := (I^{(*)})^{(*)} = I^{(*)}$. We summarize in the following
phrased in terms of subsets $S\subseteq R$, but it also holds when
restricting to ideals of $R$.
\begin{proposition}
\label{prp:closure}
For each set $S\subseteq R$ the set $S^{(*)}$ is a power-closed ideal and the
operation $S\mapsto S^{(*)}$ is a closure operator:
(i) $S\subseteq S^{(*)}$ for each subset $S$;
(ii)~$S\subseteq T$ implies $S^{(*)}\subseteq T^{(*)}$; and 
(iii) $S^{(**)} := (S^{(*)})^{(*)} = S^{(*)}$ for each $S$.
The set of power-closed ideals is a complete sublattice of the (complete)
lattice of ideals of $R$:  If $\cal S$ is a family of power-closed ideals
of $R$ then both $\bigcap_{S\in {\cal S}} S$ and
$\sum_{S\in {\cal S}} S$ are power-closed ideals.
\end{proposition}
For any subsets $S,T\subseteq R$ one can always define the sum
$S + T = \{f + g : f\in S,\ \ g\in T\}$ since $R$ is a ring and closed
under addition. By the above Proposition~\ref{prp:closure}
we have $S^{(*)} + T^{(*)}\subseteq (S + T)^{(*)}$ for any subsets
$S,T\in R$ and generally also
$\sum_{S\in {\cal S}}S^{(*)} \subseteq \left(\sum_{S\in {\cal S}} S\right)^{(*)}$
for any family ${\cal S}$ of subsets of $R$.

Conversely, since for any $f,g\in R$ we have $(f+g)^{(i)} = f^{(i)} + g^{(i)}$
we have $(S + T)^{(*)} \subseteq S^{(*)} + T^{(*)}$ for any subsets
$S,T\subseteq R$ and therefore the following.
\begin{observation}
\label{obs:closure-sum}
The closure operator $S\mapsto S^{(*)}$ commutes with the sum of sets
in $R$; that is 
$\left(\sum_{S\in {\cal S}} S\right)^{(*)} = \sum_{S\in {\cal S}}S^{(*)}$
for any family ${\cal S}$ of subsets of $R$.
\end{observation}

As with the sum we also have $(fg)^{(i)} = f^{(i)}g^{(i)}$ for any $f,g\in R$.
Suppose both $I$ and $J$ are power-closed ideals and consider the
power-closure $(IJ)^{(*)}$ of the product ideal $IJ$. If $F\in (IJ)^{(*)}$,
then there are $f_i\in I$ and $g_i\in J$ such that
$F = \sum_{i}(f_ig_i)^{(k_i)}$. (Note that in this sum, some elements
$f_i$ and $g_i$ might be repeated with different $k_i$.)
Since $I$ and $J$ are power-closed we get
\[
F = \sum_{i}(f_ig_i)^{(k_i)} = \sum_{i}f_i^{(k_i)}g_i^{(k_i)} \in IJ.
\]
\begin{observation}
\label{obs:IJ-p-closed}
For power-closed ideals $I$ and $J$ of $R$ we have that
$IJ$ is also power-closed. In particular, for any $n\in\nats$
the ideal $I^n$ is also power-closed.
\end{observation}
{\sc Remark:} For ideals $I$ and $J$ of $R$ we clearly have the
containment $(I\cap J)^{(*)} \subseteq I^{(*)}\cap J^{(*)}$.
However, we will see in Observation~\ref{obs:cap-sum} in
the next Section~\ref{sec:explicit} that
unlike the sum, the intersection does not commute with power-closure.
That is, we do not have the equality $(I\cap J)^{(*)} = I^{(*)}\cap J^{(*)}$
in $R$ in general.

The collection of ideals of $R$, partially ordered by inclusion,
forms a complete lattice in which the meet operation is intersection and the
join operation is sum.
We conclude this section with a few semi-dual observations about
power-closed ideals in this lattice.

For a given ideal $I$ of $R$, let
\[
{\cal{MC}}(I) = \{J\subseteq I : \mbox{$J$ is a power-closed ideal of $R$}\}.\]
 If 
\begin{equation}
\label{eqn:MC-above}  
I^{{(\circ)}} := \sum_{J\in{\cal{MC}}(I)}J
\end{equation}
we see that $I^{{(\circ)}}$ is the the unique maximum power-closed
ideal contained in
$I$. Call $I^{{(\circ)}}$ the {\em power-interior} of the ideal $I$.

The power-interior ideal can also be given constructively ``from below'':
If $f_1^{(i)},\ldots,f_n^{(i)}\in I$ for each $i\in\nats$,
then by analogous argument to prove Claim~\ref{clm:*-gen} we have that 
$(f_1,\ldots,f_n)^{(*)}\subseteq I$. Considering a general
element of the power-interior of $I$ as defined in (\ref{eqn:MC-above})
we have the following.
\begin{observation}
\label{obs:MC-below}
For an ideal $I$ of $R$, the power-interior ideal $I^{{(\circ)}}$ of $I$
as defined in (\ref{eqn:MC-above}) can be given as
$I^{{(\circ)}} = \{f\in I : f^{(i)}\in I \mbox{ for each $i$ }\}$.
\end{observation}
Note that since $R$ is Noetherian
$I^{{(\circ)}} = (f_1,\ldots,f_n)$ for finitely many elements
$f_1,\ldots,f_n\in I$ satisfying
$f_1^{(i)},\ldots,f_n^{(i)}\in I$ for each $i\in\nats$.

By Observation~\ref{obs:MC-below} we have for every element $\overline{f}$
in the quotient ring $R/I^{{(\circ)}}$ that $\overline{f} = 0$ implies that
$\overline{f^{(i)}} = 0$ for every $i\in\nats$. The ring $R/I^{{(\circ)}}$ is
the universal quotient ring of $R$ with this property.
\begin{proposition}
\label{prp:universal}
For an ideal $I$ of $R$, every ring homomorphism $\phi : R\rightarrow R/I$
such that $\phi(f) = 0$ implies $\phi(f^{(i)}) = 0$ for each $i\in\nats$
factors through $R/I^{{(\circ)}}$. That is, there is a unique homomorphism
$\overline{\phi} : R/I^{{(\circ)}}\rightarrow R/I$ such that
$\phi = \overline{\phi}\circ\pi$, where
$\pi : R\twoheadrightarrow R/I^{{(\circ)}}$
is the natural surjection.
\end{proposition}
The power-interior operation $I\mapsto I^{(\circ)}$ does satisfy many
of the usual interior operation for open sets in topology. It is routine
to verify that when we dually replace ``$(*)$'' with ``$(\circ)$'',
``$\subseteq$'' with  ``$\supseteq$'', and ``$+$'' with ``$\cap$'' in 
Proposition~\ref{prp:closure},
then we have valid statements. We iterate this and summarize in the
following.
\begin{proposition}
\label{prp:replace}
For each ideal $I$ the ideal $I^{(\circ)}$ is power-closed and the
operation $I\mapsto I^{(\circ)}$ is an interior operator:
(i) $I\supseteq I^{(\circ)}$ for each ideal $I$;
(ii)~$I\supseteq J$ implies $I^{(\circ)}\supseteq J^{(\circ)}$ for ideals
$I$ and $J$ of $R$. 
(iii) We have the idempotent property
$I^{(\circ\circ)} := (I^{(\circ)})^{(\circ)} = I^{(\circ)}$ for each ideal $I$.
(iv) The interior operator $I\mapsto I^{(\circ)}$ commutes with the intersection
of ideals in $R$; that is
$\left(\bigcap_{I\in {\cal I}} I\right)^{(\circ)}
= \bigcap_{I\in {\cal I}}I^{(\circ)}$
for any family ${\cal I}$ of ideals of $R$.
\end{proposition}
{\sc Remark:} For ideals $I$ and $J$ of $R$ we clearly have
$I^{(\circ)} + J^{(\circ)}\subseteq (I+J)^{(\circ)}$. However, we do not
have the other containment in general as we will see in the following example.
\begin{example}
\label{exa:IplusJ-circ}  
Consider $\alpha\in\comps$ such that
$\langle\alpha\rangle = \{\alpha^n : n\in\ints\}$ forms an infinite
multiplicative subgroup of ${\comps}^*$. If
$I = (x - \alpha)\subseteq\comps[x]$ and $J = (x + \alpha)\subseteq\comps[x]$,
then $I + J = \comps[x] = R$ and hence $(I+J)^{(\circ)} = \comps[x]$ as well.
By definition we have
\[
I^{(\circ)} = \{f \in I : f^{(i)}\in I\} = \{f\in\comps[x] :
f(\alpha^i) = 0 \mbox{ for all }i\in\nats\},
\]
and so $I^{(\circ)} = \{0\}$ since $\langle\alpha\rangle$ is infinite.
Similarly we have $J^{(\circ)} = \{0\}$ and hence
$I^{(\circ)} + J^{(\circ)} = \{0\} \neq \comps[x]$.
\end{example}
{\sc Remark:} For any ideal $I$ of $R$ we have that
$I^{{(\circ)}}\subseteq I\subseteq I^{(*)}$, so $I$ is always between two
power-closed ideals. The ideal $I$ is power-closed if and only
if $I^{{(\circ)}} = I^{(*)}$, in which case they both equal $I$.

As for the compositions of the power-closure and power-interior operations
we have the following directly by definitions.
\begin{observation}
\label{obs:close-int}  
For each ideal $I$ of $R$ we have
\[
I^{(\circ *)} := (I^{(\circ)})^{(*)} = I^{(\circ)}, \ \
I^{(* \circ)} := (I^{(*)})^{(\circ)} = I^{(*)}.
\]
\end{observation}

\section{Some explicit results}
\label{sec:explicit}

As the title of this section indicates, we will delve into some refinements
of what what was observed in the previous section. In particular, we will
strengthen Claim~\ref{clm:*-gen} among other results. 

Recall the elementary symmetric functions of
$R = \comps[x_1,\ldots,x_d]$~\cite[p.~252]{Hungerford}:
\begin{eqnarray*}
\sigma_0(d) & = & 1, \\
\sigma_1(d) & = & \sum_{i=1}^dx_i, \\
\sigma_2(d) & = & \sum_{1\leq i<j\leq d}x_ix_j, \\
  &\vdots & \\
\sigma_k(d) & = & \sum_{1\leq i_1<\cdots<i_k\leq d}x_{i_1}\cdots x_{i_k},\\
  &\vdots & \\
\sigma_d(d) & = & x_1\cdots x_d.
\end{eqnarray*}
We also put $\sigma_i(d) = 0$ for $i>d$. Whenever
 $i, d \ge 1$ 
 we have the relation
$\sigma_i(d) = x_d\sigma_{i-1}(d-1) + \sigma_i(d-1)$.
Using these relations we get
\begin{eqnarray*}
\sum_{i=0}^{d-1}(-1)^ix_d^{d-i}\sigma_{i+1}(d) & = &
\sum_{i=0}^{d-1}(-1)^ix_d^{d-i}[x_d\sigma_i(d-1) + \sigma_{i+1}(d-1)] \\
& = & \sigma_0(d-1)x_d^{d+1} \\
& = & x_d^{d+1},
\end{eqnarray*}
and hence by symmetry we get for any $\ell\in\{1,\ldots,d\}$
that
\[
\sum_{i=0}^{d-1}(-1)^ix_{\ell}^{d-i}\sigma_{i+1}(d) = x_{\ell}^{d+1}
\]
and therefore the following.
\begin{claim}
\label{clm:x^N}
For any $\ell\in\{1,\ldots,d\}$ and $n\geq d+1$ we have
\[
x_{\ell}^n = \sum_{i=0}^{d-1}(-1)^ix_{\ell}^{n-i-1}\sigma_{i+1}(d).
\]
\end{claim}
From this we obtain the following.
\begin{lemma}
\label{lmm:symm}
If $f = a_1x_1 + \cdots + a_dx_d\in R$ then for all $n\geq d+1$
we have $f^{(n)} \in (f^{(n-d)},\ldots,f^{(n-1)})$ and
therefore $f^{(n)} \in (f,f^{(2)},\ldots,f^{(d)})\subseteq R$.
In particular, $(f)^{(*)} = (f,f^{(2)},\ldots,f^{(d)})$.
\end{lemma}
{\sc Remark:} Note that Claim~\ref{clm:x^N} yields
Newton's Identities~\cite{Newton-wiki} by summing both the left
and the right hand side when $\ell$ goes from $1$ to $d$.
\begin{example}
\label{exa:yax}  
Suppose $\alpha\in \comps$ is a constant that is nonzero and
not a root of unity. By the above Lemma~\ref{lmm:symm} the ideal
 $I^{(*)} = (y-\alpha x,y^2-\alpha x^2)\subseteq R = {\comps}[x,y]$ 
is power-closed. We now argue that $I^{(*)}$ cannot be generated by 
products of monomials and binomials. Using the 
degree lexicographical order (or {\em deglex} for short) 
on $[x,y]$ for which $x < y$,  we obtain a 
Gr\"{o}bner basis  $B = \{y -\alpha x, x^2\}$ for
the ideal $I^{(*)}$ of $\comps[x,y]$. There is an associated system
of reductions
\begin{equation}
\label{eqn:red}
x^2 \mapsto 0, \ \ y\mapsto\alpha x
\end{equation}
for the elements of the quotient ring $R/I^{(*)} = \comps[x,y]/I^{(*)}$,
which is, in particular, a two dimensional vector space over $\comps$
spanned by images of $1$ and $x$.

Suppose that $(y-\alpha x, x^2) = I^{(*)} = (g_1,\ldots,g_m)$
where each $g_i$ is a product of a monomial and binomials.
Since $g_i\in I^{(*)}$ for each $i$ it must reduce to zero by 
the reductions in (\ref{eqn:red}) from $B$. Such reduction can
be done by first reducing each $y$ in each $g_i$ to $\alpha x$
and then reducing each $x^2$ to zero.

For $f\in\comps[x,y]$, $f \ne 0$, let $\delta(f)$ denote
the minimum total degree
of monomials in the expression of $f$ as a linear combination of monomials;
that is, writing $f(x,y) = \sum_{i=1}^kc_ix^{a_i}y^{b_i}$, where $c_1 \ne 0$ and
$a_1+b_1 \leq\cdots\leq a_k+b_k$, we have  $\delta(f) = a_1+b_1$.
 An element $g\in\comps[x,y]$ that is a product of
monomials and binomials has the form
\[
g = \pm x^ey^f\prod_{i=1}^k(x^{a_i}y^{b_i} - x^{c_i}y^{d_i})
\]
where we may arrange that $a_i+b_i\leq c_i+d_i$ for
each $i$, in which case we have $\delta(g) = e+f + \sum_{i=1}^k(a_i+b_i)$.
For each $f\in\comps[x,y]$ denote by $r(f)$ the polynomial in $\comps[x]$
obtained by reducing each $y$ in $f$ to $\alpha x$.
Since $\alpha$ is neither $0$ nor a root of unity,
we note that for none of the elements $g_i$ do
we have $r(g_i) = 0$ in $\comps[x]$. Further we have
$\delta(g_i) = \delta(r(g_i))$ for each $i\in\{1,\ldots,n\}$.
Since each $g_i\in I^{(*)} = (y-\alpha x,x^2)$ we have that
$\delta(r(g_i)) \geq 2$. Since $y-\alpha x\in (g_1,\ldots,g_n)$
we have $y-\alpha x = \sum_{i=1}^nf_ig_i$ for some $f_i\in\comps[x,y]$;
and we obtain a contradiction:
\[
1 = \delta(y-\alpha x) = \delta\left(\sum_{i=1}^nf_ig_i\right)
\geq 2.
\]
Hence, the ideal $I^{(*)} = (y-\alpha x,x^2)$
cannot be generated by $g_1,\ldots,g_n$.
\end{example}
We conclude with the following.
\begin{proposition}
\label{prp:prod-binom}
For $d\geq 2$, not all power-closed ideals in $\comps[x_1,\ldots,x_d]$
are generated by products of monomials and binomials.
\end{proposition}
The above Proposition~\ref{prp:prod-binom} also holds when $d = 1$. However,
the case $d=1$ is special since ${\comps}[x]$ is a PID. We will discuss that
in more detail in Section~\ref{sec:one-var}.

For $f = \sum_{i=1}^ka_i\tilde{x}^{\tilde{p}_i}$ written as minimum linear
combination of monomials in $[x_1,\ldots,x_d]$,
let $\lambda(f) = k$. By substituting
$x_i\leftarrow \tilde{x}^{\tilde{p}_i} = x_1^{p_{i\/1}}\cdots x_d^{p_{i\/d}}$ 
for each $i$, we get by Lemma~\ref{lmm:symm} the following strengthening of 
Claim~\ref{clm:*-gen}.
\begin{proposition}
\label{prp:*-gen}
For an ideal $I = (f_1,\ldots,f_n)\subseteq R$ we have 
\[
I^{(*)} = (
f_1,f_1^{(2)},\ldots,f_1^{(\lambda(f_1))},f_2,f_2^{(2)},
\ldots,f_2^{(\lambda(f_2))}, \ldots,
f_n,f_n^{(2)},\ldots,f_d^{(\lambda(f_n))}).
\]
\end{proposition}
\begin{example}
\label{exa:yax-ybx}
Let $I = (y - \alpha x)$ and
$J = (y - \beta x)$ be ideals of $\comps[x,y]$ where $\alpha\neq\beta$
be ideals as in the previous Example~\ref{exa:yax}.
Then, as before,
we have $I^{(*)} = (y - \alpha x, x^2)$ and
$J^{(*)} = (y - \beta x, x^2)$. Using reductions as in described above in
(\ref{eqn:red}) it is easy to see that $I^{(*)}\cap J^{(*)} = (x,y)^2$.
Since $\comps[x,y]$ is a unique factorization domain and 
$\alpha\neq\beta$, we have that $y -\alpha x$ and $y- \beta x$ are
relatively prime in $\comps[x,y]$. Therefore
$I\cap J = IJ = ((y -\alpha x)(y - \beta x))$ and hence by
Proposition~\ref{prp:*-gen} we get that
\[
(I\cap J)^{(*)} = ((y -\alpha x)(y - \beta x),
(y^2 -\alpha x^2)(y^2 - \beta x^2)),(y^3 -\alpha x^3)(y^3 - \beta x^3)).
\]
By looking at the total degree two
part of both $x^2$ and a general element of $(I\cap J)^{(*)}$ as in the
above display, we see that $x^2\not\in (I\cap J)^{(*)}$ and hence
$(I\cap J)^{(*)}\neq (x,y)^2 = I^{(*)}\cap J^{(*)}$.
\end{example}
\begin{observation}
\label{obs:cap-sum}
The $(*)$ operation on ideals in $\comps[x_1,\ldots,x_d]$
does not in general commute with intersection when $d\geq 2$;
there are ideals $I$ and $J$ of $\comps[x_1,\ldots,x_d]$ with
$(I\cap J)^{(*)}\neq I^{(*)}\cap J^{(*)}$.
\end{observation}

Consider again a principal ideal $I = (f)$ where
$f = a_1x_1 + \cdots + a_dx_d\in R$. By Lemma~\ref{lmm:symm}
we have $I^{(*)} = (f)^{(*)} = (f,f^{(2)},\ldots,f^{(d)})$.
We can assume that $a_d\neq 0$. By using a version
of Gauss-Jordan elimination/reduction on the elements
$f, f^{(2)},\ldots, f^{(d)}$ when we view the variables
$x_1,\ldots, x_d$ as the coefficients we get
\begin{equation}
  \label{eqn:linear-ideal}
  \begin{split}
    (f)^{(*)} = & \left( a_1x_1 + a_2x_2 + a_3x_3 + \cdots + a_dx_x, \right. \\
    & \left. a_2x_2(x_2-x_1) + a_3x_3(x_3-x_1) + \cdots a_dx_d(x_d-x_1),
    \right. \\
    & \left. a_3x_3(x_3-x_1)(x_3-x_2) +\cdots a_dx_d(x_d-x_1)(x_d-x_2) \right. \\
    & \left. \vdots \right. \\
    & \left. a_dx_d(x_d-x_1)(x_d-x_2)\cdots(x_d-x_{d-1}) \right). 
    \end{split}
  \end{equation}
By the same token we have the following.
\begin{observation}
\label{obs:principal}  
For a general principal ideal $I = (f)$
where $f = \sum_{i=1}^ka_i\tilde{x}^{\tilde{p}_i}\in R$ we have that
\[
\begin{split}
  (f)^{(*)} = & \left( \sum_{i=1}^ka_i\tilde{x}^{\tilde{p}_i}, \right. \\
  & \left. \sum_{i=2}^ka_i\tilde{x}^{\tilde{p}_i}
  (\tilde{x}^{\tilde{p}_i} - \tilde{x}^{\tilde{p}_1}), \right. \\
  & \left. \sum_{i=3}^ka_i\tilde{x}^{\tilde{p}_i}
  (\tilde{x}^{\tilde{p}_i} - \tilde{x}^{\tilde{p}_1})
  (\tilde{x}^{\tilde{p}_i} - \tilde{x}^{\tilde{p}_2}), \right. \\
  & \left. \vdots \right. \\
  & \left. \vphantom{\sum_{i=1}^ka_i\tilde{x}^{\tilde{p}_i}}
  a_k\tilde{x}^{\tilde{p}_k}(\tilde{x}^{\tilde{p}_k} - \tilde{x}^{\tilde{p}_1})
  (\tilde{x}^{\tilde{p}_k} - \tilde{x}^{\tilde{p}_2})\cdots
  (\tilde{x}^{\tilde{p}_k} - \tilde{x}^{\tilde{p}_{k-1}})\right).
  \end{split}
  \]
\end{observation}
Consider again $f = a_1x_1 + \cdots + a_dx_d\in R$.
Suppose we want to describe the set of zeroes
$\tilde{x} = (x_1,\ldots,x_d)$ of the ideal 
$I^{(*)} = (f,f^{(2)},\ldots,f^{(d)})$. We still can assume that $a_d\neq 0$.
Writing the equations
as $f = 0, f^{(2)} = 0, \ldots, f^{(d)} = 0$ and viewing
the variables $x_i$ as coefficients we obtain an equivalent
system of linear equations
\begin{equation}
\label{eqn:Vandermonde}
V(\tilde{x})^tD(\tilde x)\tilde{a}^t = \tilde{0}^t,
\end{equation}
where $V(\tilde{x})$ denotes the $d\times d$ Vandermonde matrix in 
terms of the variables
$\tilde{x} = (x_1,\ldots,x_d)$ and $D(\tilde x)$ is the diagonal matrix
having diagonal entries equal to those of $\tilde{x}$. By assumption,
 $\tilde{a} \neq \tilde{0}$, so the determinant of
$V(\tilde{x})D(\tilde x)$ must be zero, implying that either
some $x_i$ is zero or, for some pair $i \ne j$, $x_i = x_j$.
Using (\ref{eqn:linear-ideal}) above, a stronger statement is obtained
from the following system:
\begin{eqnarray*}
a_1x_1 + a_2x_2 + a_3x_3 + \cdots + a_dx_x & = & 0 \\
a_2x_2(x_2-x_1) + a_3x_3(x_3-x_1) + \cdots a_dx_d(x_d-x_1) & = & 0 \\
a_3x_3(x_3-x_1)(x_3-x_2) +\cdots a_dx_d(x_d-x_1)(x_d-x_2) & = & 0 \\
  & \vdots & \\
a_dx_d(x_d-x_1)(x_d-x_2)\cdots(x_d-x_{d-1}) & = & 0.
\end{eqnarray*}
By assumption we have $a_d\neq 0$, so the last equation
implies that $x_d(x_d-x_1)(x_d-x_2)\cdots(x_d-x_{d-1}) = 0$. 
Hence, either $x_d = 0$ or $x_d = x_i$ for some $i\in\{1,\ldots,d-1\}$.
By induction on $d$ and based on what the coefficients
$a_1,\ldots,a_d\in {\comps}$ are, we can then deduce the following
about the {\em zero locus}, or the {\em affine variety},
$Z(J) = \{\tilde w \in (\comps^*)^d : f(\tilde w) = 0$ for each $f\in J\}$;
of an ideal $J$, namely the following.
\begin{proposition}
\label{prp:Z(f*)}
For $f = a_1x_1 + \cdots + a_dx_d\in R$ and $I = (f)$
we have $I^{(*)} = (f,f^{(2)},\ldots,f^{(d)})$ and there
is an antichain $\mathcal{A}\subseteq 2^{[d]} = \power(\{1,\ldots,d\})$
such that $\sum_{i\in A}a_i  = 0$ for each $A\in\mathcal{A}$ and
\[
Z(I^{(*)}) = 
\bigcup_{A\in\mathcal{A}}\left\{\sum_{i\in A}t\tilde{e}_i : t\in \comps \right\}
= \bigcup_{A\in\mathcal{A}}\comps\left(\sum_{i\in A}\tilde{e}_i\right)
\]
a union of lines in $\comps ^d$ with pairwise distinguishing $1$-projections.
In particular, if no partial sum $\sum_{i\in A}a_i$ equals zero for any
$A\subseteq [d]$, then $Z(I^{(*)}) = \{\tilde{0}\}$.
\end{proposition}

\section{Laurent polynomials}
\label{sec:Laurent}

In this section we briefly consider the corresponding Laurent polynomial ring
over the complex number field
$R^{\pm}= \comps[x_1,\ldots,x_d]^{\pm} = \comps[x_1^{\pm 1},\ldots,x_d^{\pm 1}]$
in the variables $x_1,\ldots,x_n$ and their multiplicative inverses
$x_1^{-1},\ldots,x_n^{-1}$. Note that as an image of 
$\comps[x_1,\ldots,x_n,y]$, namely
$R^{\pm}\cong \comps[x_1,\ldots,x_d,y]/(x_1\cdots x_dy - 1)$,
then $R^{\pm}$ is Noetherian. As a localization $S^{-1}R$ where $S$ is the
multiplicatively closed semigroup $S = [x_1,\ldots,x_d]$
it is also a UFD. From Claim~\ref{clm:x^N} we get by dividing 
through by an arbitrary power of $x_{\ell}$ that
\[
x_{\ell}^n = \sum_{i=0}^{d-1}(-1)^ix_{\ell}^{n-i-1}\sigma_{i+1}(d)
\]
holds for every $n\in\ints$. Therefore for any fixed $k\in\ints$,
$f = a_1x_1 + \cdots + a_dx_d\in R^{\pm}$, $I = (f)$ and  
$I_k = (f^{(k)},\ldots,f^{(k+d-1)})$ then, as in Lemma~\ref{lmm:symm},
$f^{(n)}\in I_k$ for any $n\geq k$.

Also by Claim~\ref{clm:x^N} we get 
\[
x_{\ell}^{-n} = 
\sum_{i=1}^d \frac{(-1)^i\sigma_{d-i}(d)}{\sigma_d(d)}x_{\ell}^{i-n}
\]
for each $n\in\ints$ and hence $f^{(n)}\in I_k$ for each $n\leq k$, 
and therefore $f^{(n)}\in I_1 = (f,f^{(2)},\ldots,f^{(d)})$
for each $n\in\ints$. We summarize in the following.
\begin{lemma}
\label{lmm:symm-pm}
If $f = a_1x_1 + \cdots + a_dx_d\in R^{\pm}$ and
$I_k = (f^{(k)},\ldots,f^{(k+d-1)})$ for some $k\in\ints$,
then $f^{(n)}\in I_k$ for all $n\in\ints$.
\end{lemma}
{\sc Note:} For $I_k\subseteq R^{\pm}$ a proper ideal,
then we must have $\sum_{i=1}^da_i = 0$.

As a corollary of Lemma~\ref{lmm:symm-pm} we then have the following.
\begin{corollary}
\label{cor:pm-gen}
For any $f \in R^{\pm}$ and $I_k = (f^{(k)},\ldots,f^{(k+\lambda(f)-1)})$ 
for some $k\in\ints$, then $f^{(n)}\in I_k$ for each $k\in\ints$.
\end{corollary}
Further we have the following generalization directly
from Corollary~\ref{cor:pm-gen}.
\begin{theorem}
\label{thm:pm-gen2}
For an ideal $I = (f_1,\ldots,f_n) \subseteq R^{\pm}$ and for
any integers $k_1,\ldots,k_d\in\ints$ we have
\[
I^{(*)} = (
f_1^{(k_1)},\ldots,f_1^{(k_1 + \lambda(f_1)-1)},
f_2^{(k_2)},\ldots,f_2^{(k_2 + \lambda(f_2)-1)}, 
\ldots,
f_n^{(k_n)},\ldots,f_d^{(k_n + \lambda(f_n)-1)}).
\]
\end{theorem}
From Theorem~\ref{thm:pm-gen2} we get the following.
\begin{observation}
\label{obs:ints}
For an ideal $I = (f_1,\ldots,f_n) \subseteq R^{\pm}$ we
have $f_{\ell}^{(i)}\in I^{(*)}$ for each $\ell\in\{1,\ldots,d\}$ 
and each $i\in\ints$. In particular, if $f\in I\subseteq R^{\pm}$ and 
$I = I^{(*)}$ is power-closed, then $f^{(i)}\in I$ for each $i\in\ints$.
\end{observation}

As with the usual polynomial ring $R = \comps[x_1,\ldots,x_n]$, we have
analogously that each ideal 
$I\subseteq R^{\pm}= \comps[x_1^{\pm 1},\ldots,x_d^{\pm 1}]$ generated 
by products of binomials in $I = (g_1,\ldots,g_n)$ is a power-closed
ideal of $R^{\pm}$. Also, in a similar fashion as for the polynomial
ring $R$, there are power-closed ideals of $R^{\pm}$ that are not
generated by products of binomials, as we will now show.
First note that in Example~\ref{exa:yax} for $d\geq 2$,
the ideal $I = (y-\alpha x)\subseteq \comps[x,y]^{\pm}$
we have by Observation~\ref{obs:ints} that
$I^{(*)}= R^{\pm}$ whenever $\alpha\neq 1$ so we cannot
use the same ideal as that Example~\ref{exa:yax}.
We will rather use a variant of it. Again, the case $d=1$ is special
and we will discuss that in detail in Section~\ref{sec:one-var}.
\begin{example}
\label{exa:zaxby}  
For given nonzero constants $\alpha,\beta\in\comps$
consider the ideal 
$I = (z - \alpha x - \beta y)\subseteq \comps[x,y,z]^{\pm}$,
where we assume $\alpha + \beta = 1$.
By Theorem~\ref{thm:pm-gen2} we obtain its power-closure
\[
I^{(*)} = 
(z - \alpha x - \beta y,z^2 - \alpha x^2 - \beta y^2,
z^3 - \alpha x^3 - \beta y^3).
\]
Since $\alpha+\beta=1$, this ideal is contained in the maximal ideal
$(x-1,y-1,z-1)$; so, $I^{(*)} \ne \comps[x,y,z]^\pm$.
By isolating $z$ in the first generator and writing the other
two in terms of $x$ and $y$, we obtain that
$I^{(*)} = (z - \alpha x - \beta y, (y-x)^2)$. Since
$z - \alpha x -\beta y\in I^{(*)}$ we have by
Observation~\ref{obs:ints} that $z^k - \alpha x^k - \beta y^k\in I^{(*)}$
for each $k\in\ints$. 

Assume that $I^{(*)}$  is generated by products
of binomials in $\comps[x,y,z]^{\pm}$.  The binomials may be
taken to be  of the form $x^ay^bz^c - 1$ where $a,b,c\in\ints$, so we may
write $I^{(*)} = (g_1, \ldots, g_n)$, where each $g_i$ is a product of
binomials of the form $x^ay^bz^c - 1$.
Since each $g_i \in (z - \alpha x - \beta y, (y-x)^2)$ then
$g_i \equiv g_i' \pmod{I^{(*)}}$ where $g_i'$ is the Laurent polynomial
obtained from $g_i$ by replacing each $z^c$ with $\alpha x^c + \beta y^c$,
and so is a product of Laurent polynomials of the form
$x^ay^b(\alpha x^c + \beta y^c) - 1$. The following is easy
to verify.
\begin{claim}
\label{clm:ab-irrational}
For $a,b,c\in\ints$, $\alpha, \beta\in\comps$ irrational satisfying
$\alpha + \beta = 1$ we have $x^ay^b(\alpha x^c + \beta y^c) - 1$ is
(i) divisible by $y-x$ in $\comps[x,y,z]^{\pm}$ if and only if
$a + b + c = 0$, and
(ii) never divisible by $(y-x)^2$.
\end{claim}
Since each $g_i'$ is divisible by $(y-x)^2$ for each $i$, then
if we further assume both $\alpha$ and $\beta$ to be irrational,
we have by Claim~\ref{clm:ab-irrational} that 
$g_i'$ has at least two factors of the form
$x^ay^b(\alpha x^c + \beta y^c) - 1$ where $a + b + c = 0$.
By assumption $(z - \alpha x - \beta y, (y-x)^2) = (g_1,\ldots,g_n)$
and so $z - \alpha x - \beta y \in (g_1,\ldots,g_n)$, so
$z - \alpha x - \beta y = \sum_{i=1}^n f_ig_i$ for some
$f_1,\ldots,f_n\in \comps[x,y,z]^{\pm}$. By letting $z = y$
we then obtain an identity
\begin{equation}
\label{eqn:xyy}
\alpha(y-x) = \sum_{i=1}^n f_i(x,y,y)g_i(x,y,y),
\end{equation}
in $\comps[x,y]^{\pm}$ where each $g_i$ has as at least two
factors of the form $x^ay^by^c - 1 = x^ay^{b+c} - 1$ where $a+b+c = 0$.
Since $x^ay^{b+c} - 1 = y^{-a}(x^a - y^a)$ is divisible by $y-x$
then the right hand side of (\ref{eqn:xyy}) is always divisible by
$(y-x)^2$ and so $\alpha(y-x) = (y-x)^2h(x,y)$ for some
$h\in\comps[x,y]^{\pm}$ which is a contradiction, since
$\comps[x,y]^{\pm}$ is a UFD with primes from those of $\comps[x,y]$
that are not contained in the semigroup $[x,y]$. Hence,
$I^{(*)} = (z - \alpha x - \beta y, (y-x)^2)$ is power-closed
in $\comps[x,y,z]^{\pm}$ and is not generated by products of binomials.
\end{example}
\begin{example}
\label{exa:zaxby-cont}  
Continuing with the ideal
$I = (z - \alpha x - \beta y)\subseteq \comps[x,y,z]^{\pm}$ from
Example~\ref{exa:zaxby} here above, where $\alpha + \beta = 1$, and let
$J = (z - \gamma x - \delta y)\subseteq \comps[x,y,z]^{\pm}$ be another
such ideal also with $\gamma + \delta = 1$. From the above
Example~\ref{exa:zaxby} we have $I^{(*)} = (z - \alpha x - \beta y, (y-x)^2)$,
$J^{(*)} = (z - \gamma x - \delta y, (y-x)^2)$ and hence
$(y-x)^2 \in I^{(*)}\cap J^{(*)}$. 

Since $R^{\pm} = \comps[x,y,z]^{\pm}$ is a UFD and $z - \alpha x - \beta y$,
$z - \gamma x - \delta y$ are two distinct irreducible elements in $R^{\pm}$
then $I\cap J = IJ = ((z - \alpha x - \beta y)(z - \gamma x - \delta y))$.
By Theorem~\ref{thm:pm-gen2} we obtain its power-closure
\[
(I\cap J)^{(*)} =
((z - \alpha x - \beta y)(z - \gamma x - \delta y),\ldots,
(z^6 - \alpha x^6 - \beta y^6)(z^6 - \gamma x^6 - \delta y^6))
\]
as an ideal in $R^{\pm}$. We now argue that $(y-2)^2\not\in (I\cap J)^{(*)}$
as follows:

Suppose $(y-2)^2\in (I\cap J)^{(*)}$. In that case we have an equation
\[
(y-x)^2 = \sum_{i=1}^6f_i(z^i - \alpha x^i
- \beta y^i)(z^i - \gamma x^i - \delta y^i),
\]
for some $f_1,\ldots,f_6 \in R^{\pm}$. Evaluating at $z = 1$ we the
obtain an equation
\[
(y-x)^2 = \sum_{i=1}^6f_i(1 - \alpha x^i
- \beta y^i)(1 - \gamma x^i - \delta y^i),
\]
for some $f_1,\ldots,f_6 \in {\comps}[x,y]^{\pm}$\footnote{Note that
  these two equations are equivalent since we can obtain the first
  one by homogenizing the second one.}.
We note that each $f_i$ can be assumed to be a sum of monomials in
$x$ and $y$ where the degree in each of the two variables is within
the range $\{-12,-11,\ldots,1,2\}$, and hence we obtain an equation
\begin{equation}
\label{eqn:Macaulay2}
(xy)^{12}(y-x)^2 = \sum_{i=1}^6g_i(1 - \alpha x^i
- \beta y^i)(1 - \gamma x^i - \delta y^i),
\end{equation}
where each $g_i = (xy)^{12}f_i\in{\comps}[x,y]$ is a polynomial.
By a quick computation using the software package Macaulay2~\cite{Macaulay2},
which uses an incorporated Gr\"{o}bner basis, yields that
\[
(xy)^{12}(y-x)^2\not\in 
((z - \alpha x - \beta y)(z - \gamma x - \delta y),\ldots,
(z^6 - \alpha x^6 - \beta y^6)(z^6 - \gamma x^6 - \delta y^6))
\]
as an ideal in the polynomial ring ${\comps}[x,y]$. Hence,
(\ref{eqn:Macaulay2}) has no solutions in polynomials
$g_1,\ldots,g_6$. We can therefore conclude that
$(y-2)^2\not\in (I\cap J)^{(*)}$ as an ideal in
$R^{\pm} = {\comps}[x,y,z]^{\pm}$ in this case.
Hence, there are ideals $I$ and $J$ in the Laurent polynomial
ring $R^{\pm}$ as well such that
$(I\cap J)^{(*)} \neq I^{(*)}\cap J^{(*)}$. 
\end{example}
Note that Example~\ref{exa:IplusJ-circ} is valid in the Laurent
polynomial ring $R^{\pm} = {\comps}[x]^{\pm}$ as well as the
polynomial ring $R = {\comps}[x]$. Example~\ref{exa:IplusJ-circ},
together with Example~\ref{exa:yax-ybx} and the above
Example~\ref{exa:zaxby-cont} yield the following where we
gather these observations.
\begin{observation}
\label{obs:strict}
In both the polynomial ring $R = {\comps}[x_1,\ldots,x_d]$ and the Laurent
ring $R^{\pm}= \comps[x_1,\ldots,x_d]^{\pm}$ we have:
(i) there are ideals $I$ and $J$ such that
$I^{(\circ)} + J^{(\circ)}\subset (I+J)^{(\circ)}$ is a strict containment, and
(ii) there are ideals $I$ and $J$ such that
$(I\cap J)^{(*)} \subset I^{(*)}\cap J^{(*)}$ is a strict containment. 
\end{observation}
Recall that by Proposition~\ref{prp:replace} we do have equalities
$I^{(*)} + J^{(*)} = (I + J)^{(*)}$ and
$(I\cap J)^{(\circ)} = I^{(\circ)} \cap J^{(\circ)}$.

Going back to general results,
Observation~\ref{obs:principal} has the following consequence.
\begin{proposition}
\label{prp:principal}
If $f = \sum_{i=1}^ka_i\tilde{x}^{\tilde{p}_i} \in R$ and
$I = (f)$ is a power-closed ideal of $R^{\pm}$, so $I = I^{(*)}$, then
$f$ divides the polynomial 
\[
(\tilde{x}^{\tilde{p}_k} - \tilde{x}^{\tilde{p}_1})
  (\tilde{x}^{\tilde{p}_k} - \tilde{x}^{\tilde{p}_2})\cdots
  (\tilde{x}^{\tilde{p}_k} - \tilde{x}^{\tilde{p}_{k-1}}).
\]
Likewise, if $I = (f)$ is a power-closed ideal of $R$, then
$f$ divides the polynomial 
\[
\tilde{x}^{\tilde{p}_k}(\tilde{x}^{\tilde{p}_k} - \tilde{x}^{\tilde{p}_1})
  (\tilde{x}^{\tilde{p}_k} - \tilde{x}^{\tilde{p}_2})\cdots
  (\tilde{x}^{\tilde{p}_k} - \tilde{x}^{\tilde{p}_{k-1}}).
\]
\end{proposition}
Consider now briefly the case of a principal power-closed 
ideal $I = (f)$, where $f = \sum_{i=1}^ka_i\tilde{x}^{\tilde{p}_i}$ 
is an irreducible element.
By Proposition~\ref{prp:principal} the expression
$\tilde{x}^{\tilde{p}_k}(\tilde{x}^{\tilde{p}_k} - \tilde{x}^{\tilde{p}_1})
  (\tilde{x}^{\tilde{p}_k} - \tilde{x}^{\tilde{p}_2})\cdots
  (\tilde{x}^{\tilde{p}_k} - \tilde{x}^{\tilde{p}_{k-1}})$
is divisible by the polynomial $f$.
We can assume we have a term order on $[x_1,\ldots,x_d]$
such that $\tilde{x}^{\tilde{p}_1} < \cdots < \tilde{x}^{\tilde{p}_d}$.
Since the leading terms of $f$ and each of the binomials
$\tilde{x}^{\tilde{p}_k} - \tilde{x}^{\tilde{p}_i}$ is 
$\tilde{x}^{\tilde{p}_k}$ we conclude that 
$f = \tilde{x}^{\tilde{p}_k} - \tilde{x}^{\tilde{p}_i}$ for some 
$i\in \{1,\ldots,k-1\}$. Hence we have the following.
\begin{observation}
\label{obs:irreducible}
If the principal ideal $I = (f)$ is generated by an irreducible
element $f\in R^{\pm}$ and is power-closed, then $f$ must itself
be a binomial.
\end{observation}
In order to investigate principal power-closed ideals that are
generated by a general reducible element, we will first consider the case
$d=1$ in detail in the following section. 

\section{The one variable case}
\label{sec:one-var}

In this section we focus on the case when $d=1$;
$R = \comps[x]$ and $R^{\pm} = \comps[x]^{\pm}$. This case is
special as it is the only case for which $R$ and $R^{\pm}$
are PIDs. Every binomial in $R$ has the form $x^n - x^m$ for some
nonnegative integers $m$ and $n$ and every binomial in
$R^{\pm}$ has, up to a unit, the form $x^n - 1$ for some $n$.

\subsection{Power-polynomials}
\label{subsec:MP}

A polynomial in $R$ or $R^{\pm}$ that generates a power-closed
ideal has some nice properties. This will be our main topic in this
subsection.
\begin{definition}
\label{def:Minks-poly}
A polynomial $f = f(x)$ in $\comps[x]\subseteq \comps[x]^{\pm}$
is {\em powered} if $f$ divides $f^{(i)}$ for every $i\in \nats$.
A powered polynomial will be called a {\em power-polynomial}.
\end{definition}
Note that if $f(x)$ divides $f(x^i)$, then for each $n\in\nats$
we have that $f(x^n)$ divides $f(x^{ni})$, and hence we have
the following claim.
\begin{claim}
\label{clm:primes}
A polynomial $f = f(x)$ in $\comps[x] \subseteq \comps[x]^{\pm}$ is
a power-polynomial
if and only if $f(x)$ divides $f(x^p)$ for every prime number $p$.
\end{claim}
\begin{example}
\label{exa:ab1}  
Let $a,b\in \nats$ be relatively prime positive integers
and let $f(x) = (x^a-1)(x^b-1)/(x-1) \in R^{\pm}$. Clearly, $f(x)$ is not
a product of binomials in $R^{\pm}$. If $p$ is a prime number, then either
$a$ or $b$ is not divisible by $p$, say $a$. Note that $x^b-1$
divides $x^{pb}-1$ and $x^a-1 =\prod_{i=0}^{a-1}(x - \rho^i)$
where $\rho\in \comps$ is a primitive $a$-th root of unity, so
$(x^a-1)/(x-1) = \prod_{i=1}^{a-1}(x - \rho^i)$.
Since $p$ does not divide $a$ then $\rho^p$ is also a primitive $a$-th
root of unity, so
$\{\rho, \rho^2, \ldots, \rho^{a-1}\}
=\{\rho^p, \rho^{2p}, \ldots, \rho^{(a-1)p}\}$.
Therefore
\[
\frac{x^{pa}-1}{x^p-1} = \prod_{i=1}^{a-1}(x^p - \rho^i)
= \prod_{i=1}^{a-1}(x^p - \rho^{pi}),
\]
which is divisible by $(x^a-1)/(x-1)$ since each $x^p - \rho^{pi}$
is divisible by $x-\rho^i$. Hence, $f(x)$ divides $f(x^p)$ for each
prime that does not divide $a$. By symmetry, if $p$ is a prime that
does not divide $b$ then also $f(x)$ divides $f(x^p)$. So for any prime
$p$, $f(x)$ divides $f(x^p)$. By Claim~\ref{clm:primes}
we therefore have that $f(x)$ is a power-polynomial.
\end{example}
The following observation therefore applies.
\begin{observation}
\label{obs:not-MP}
Both rings $R = \comps[x]$ and $R^{\pm} = \comps[x]^{\pm}$
contain power-polynomials that are not products of
binomials of the form $x^n-1$. Consequently, there are
principal power-closed ideals of both $R$ and $R^{\pm}$ that
are not generated by products of binomials.
\end{observation}
Suppose $f$ and $g$ are both powered and consider $l= \lcm(f,g)$,
their least common multiple in $R$ or $R^{\pm}$. Let 
$(x - r)^p$ be the largest power of $x-r$ occurring in $l(x)$ 
in its unique factorization into prime factors over $\comps$. 
Then $(x-r)^p$
is a factor either in $f$ or $g$, say $f$. Since $f$ is powered then
$(x-r)^p$ divides $f(x^i)$ for any positive integer $i$, and since
$f(x^i)$ divides $l(x^i)$ for any $i$, then $(x-r)^p$ divides
$l(x^i)$, for any positive integer $i$. Since this holds for any
prime factor $x-r$ of $l(x)$ we have the following.
\begin{proposition}
\label{prp:lcm}
If both $f$ and $g$ are powered, then $\lcm(f,g)$ is also powered.
\end{proposition}
{\sc Note:} Since any binomial $x^n-1$ is clearly powered, and
for any positive integers $a$ and $b$
we have $\gcd(x^a-1,x^b-1) = x^{\gcd(a,b)}-1$, we see
that for relatively prime $a$ and $b$ we have
\[
\lcm(x^a-1,x^b-1) = \frac{(x^a-1)(x^b-1)}{x^{\gcd(a,b)}-1}
= \frac{(x^a-1)(x^b-1)}{x-1} = f(x)
\]
from the above Example~\ref{exa:ab1}.
Hence, we also see from Proposition~\ref{prp:lcm}
that $f(x)$ is indeed powered.

Suppose $f$ is a power-polynomial. By
Proposition~\ref{prp:principal}, in $R^\pm$,  $f$ must divide
the polynomial $(x^{p_1} - 1)(x^{p_2} - 1)\cdots (x^{p_k} - 1)$
for some positive integers $p_1,\ldots,p_k$. We have in particular
the following.
\begin{observation}
\label{obs:roots-of-unity}
Every root of a power-polynomial $f \in R^{\pm}$ is a root
of unity.  Therefore any such polynomial is, up to multiplication
by a unit in $R^\pm$,  of the form
\[
f(x) = (x-\rho_1)^{a_1}(x-\rho_2)^{a_2}\cdots(x-\rho_k)^{a_k}
\]
for some distinct roots of unity $\rho_1,\ldots,\rho_k$ 
and  positive integers $a_1,\ldots,a_k$.
\end{observation}
Suppose $f\in R^{\pm}$ is a power-polynomial of the form described
in the Observation~\ref{obs:roots-of-unity}. For a root
$\rho$ of $f$ with multiplicity $a$ we then have that $(x-\rho)^a$ divides
$f(x^n)$ for every $n\in\nats$. Since for each $n$ the polynomials
$x^n - \rho_i$ and $x^n - \rho_j$ are relatively prime for distinct $i$ and 
$j$, there is an $i$ such that $(x-\rho)^a$ divides a factor
$(x^n - \rho_i)^{a_i}$ 
of $f(x^n)$. Since $x^n-\rho_i$ is separable, the multiplicity of $\rho_i$
in $f(x^n)$ must be at least $a$, and so $a_i\geq a$.
This means that the multiplicity of $\rho_i$ in $f(x)$ is $a_i\geq a$. So,
since $\rho^n = \rho_i$ we see that for each $n$ the multiplicity of $\rho^n$
must be at least that of $\rho$ in $f$. We summarize in the following.
\begin{claim}
\label{clm:n-multipl}
If $f\in R^{\pm}$ is powered and $\rho$ is a root of $f$, then 
for every $n\in\nats$ the multiplicity of $\rho^n$ in $f$ is at least
that of $\rho$. 

Further, if $\rho_1,\ldots, \rho_{\ell}$ are all distinct roots of a
power-polynomial $f$
with $\rho_1^n = \rho_2^n = \cdots = \rho_{\ell}^n = \rho$, then 
$(x-\rho_1)^{a_1}(x-\rho_2)^{a_2}\cdots(x-\rho_{\ell})^{a_{\ell}}$ divides
$f(x^n)$ iff the multiplicity of $x-\rho$ in $f$ is at least
$\max(a_1,a_2,\ldots,a_{\ell})$.
\end{claim}
The above Claim~\ref{clm:n-multipl} yields a necessary condition
for a power-polynomial, but also a sufficient one by considering
all possible values of $n\in\nats$ and keeping track of the multiplicities.
So we get a description of power-polynomials $f$ in one variable.
Namely, a polynomial $f\in R^{\pm}$ of the form 
$f(x) = (x-\rho_1)^{a_1}(x-\rho_2)^{a_2}\cdots(x-\rho_k)^{a_k}$,
where $\rho_1,\ldots,\rho_k$ are distinct roots of unity in $\comps$,
is a power-polynomial if and only if (i) for each $i$ we have 
$\{\rho_i,\rho_i^2, \rho_i^3,\ldots\} \subseteq \{\rho_1,\ldots,\rho_k\}$
and (ii) for each $n$ the multiplicity of $\rho_i^n$ in $f$ is at
least that of $\rho_i$.

Recall that the $n$-th cyclotomic polynomial $\phi_n$ is the product
of all the linear factors $x-\rho$ where $\rho$ is a primitive $n$-th
root of unity: $\phi_n(x) = \prod_{|\rho| = n}(x-\rho)$.
The degree of $\phi_n$ is the Euler phi function
$\phi(n)$.  Any two distinct cyclotomic polynomials are 
relatively prime. Each cyclotomic polynomial is monic, lies in
$\ints[x]$, and is irreducible over $\rats$. 

By characterization of a power-polynomial described above,
two primitive $n$-th roots of unity must have the same multiplicity
in a power-polynomial $f$. By lumping
then together in the prime factorization of $f$ we obtain another more
transparent characterization of power-polynomials.
\begin{proposition}
\label{prp:1st-char}
Up to multiplication by a unit, a polynomial $f\in R^{\pm}$ is a
power-polynomial if and only if $f$ is
a product of cyclotomic polynomials in $\ints[x]$ in such a way that 
whenever $m$ divides $n$, the multiplicity of $\phi_m$
in $f$ is at least that of $\phi_n$ in $f$. In particular, if $f$ is
monic with nonzero constant term then $f\in\ints[x]$ and the
irreducible factors of $f$ over $\rats$ are all cyclotomic polynomials.
\end{proposition}
\begin{example}
\label{exa:4342}  
Consider the polynomials $f$ and $g$ given by
\begin{eqnarray*}
f(x)  & = & \phi_{12}(x)^2\phi_{8}(x)^2\phi_{6}(x)^2\phi_{4}(x)^3
\phi_{3}(x)^2\phi_{2}(x)^3\phi_{1}(x)^4, \\
g(x)  & = & \phi_{12}(x)^2\phi_{8}(x)^3\phi_{6}(x)^2\phi_{4}(x)^2
\phi_{3}(x)^2\phi_{2}(x)^3\phi_{1}(x)^4.
\end{eqnarray*}
Here the first one $f$ is a power-polynomial, whereas the latter one $g$
is not a power-polynomial, since the multiplicity of $\phi_8$ in $g$ 
is $3$ and the multiplicity of $\phi_4$ in $g$ is $2 < 3$.
\end{example}
Using the principle of Inclusion/Exclusion (or induction on $n$)
one obtains that
\[
\phi_n(x) = \prod_{d|n}(x^d - 1)^{\mu(n/d)},
\]
where $\mu$ is the M\"{o}bius function, and so each power-polynomial
can be written as a {\em rational} expression in terms of binomials $x^n-1$
in $R^{\pm}=\comps[x]^{\pm}$. 
However, more can be said about the specific structure
of such a rational expression of polynomials in $\rats(x)$ if it is powered. 
By Proposition~\ref{prp:1st-char}
we see that for every cyclotomic factor $\phi_n$ of a power-polynomial $f$ with 
a given multiplicity, then each $\phi_d$ where $d|n$ also appears
in $f$ with at least the same multiplicity. Noting that for each $n\in\nats$
we have $\prod_{d|n}\phi_d(x) = x^n-1$, this observation yields a way
to present every power-polynomial $f$ in terms of products of
least common multiples of binomials in $R^{\pm}$. We first demonstrate
with an example.
\begin{example}
\label{exa:phi-fraction}  
We can rewrite the power-polynomial $f$ from
Example~\ref{exa:4342} as follows:
\begin{eqnarray*}
f(x) &  = & \phi_{12}(x)^2\phi_{8}(x)^2\phi_{6}(x)^2\phi_{4}(x)^3 
                \phi_{3}(x)^2\phi_{2}(x)^3\phi_{1}(x)^4 \\
  & = & (\phi_{12}(x)\phi_{8}(x)\phi_{6}(x)\phi_{4}(x) 
                \phi_{3}(x)\phi_{2}(x)\phi_{1}(x))^2\cdot
        (\phi_{4}(x)\phi_{2}(x)\phi_{1}(x))\cdot\phi_{1}(x) \\
  & = & (\lcm(x^{12} -1, x^8 - 1))^2\cdot(x^4 - 1)\cdot(x - 1).\\
  & = & \frac{(x^{12} -1)^2(x^8 - 1)^2(x^4 - 1)(x - 1)}{(x^{\gcd(12,8)} - 1)^2} \\
  & = & \frac{(x^{12} -1)^2(x^8 - 1)^2(x - 1)}{x^4 - 1},
\end{eqnarray*}
\end{example}
What is done in the above example can be done in general. Suppose
$f$ is
a monic power-polynomial with nonzero constant term.
By Proposition~\ref{prp:1st-char} it has the form
\[
f(x) = \prod_{i\in I} \phi_i(x)^{a_i},
\]
where $d|i$ implies $a_d \geq a_i$. In particular, if $d|i$ and $i\in I$
then $d\in I$ as well. Looking at the possible multiplicities
$\{ a_i : i\in I\}$ we can assume that they take $h$ distinct values,
say, $\alpha_1 < \alpha_2 < \cdots < \alpha_h$. In this case
we can write $f$ as
\begin{equation}
\label{eqn:psi}  
f(x) = \left( \prod_{i\in P_1}\phi_i(x)\right)^{\alpha_1}
  \cdot \left( \prod_{i\in P_2}\phi_i(x)\right)^{\alpha_2 - \alpha_1}
  \cdots \left( \prod_{i\in P_h}\phi_i(x)\right)^{\alpha_h - \alpha_{h-1}},
\end{equation}
where $I = P_1 \supset P_2 \supset \cdots \supset P_h$ is a nested sequence
of subsets of the index set $I$ of positive integers. 
By Proposition~\ref{prp:1st-char} it is
clear that each of the index set $P_l$ has the property that if $i\in P_l$
and $d|i$ then $d\in P_l$ as well.
\begin{definition}
\label{def:psi}  
Let $P\subseteq\nats$ be a finite subposet of $\nats$ w.r.t.~divisibility,
so whenever $i\in P$ and $d|i$ then $d\in P$ as well. A polynomial of
the following form
\[
\psi_P(x) = \prod_{i\in P} \phi_i(x),
\]
will be referred to as a {\em psi-polynomial}.
\end{definition}
Clearly, if $\max(P) = \{n_1,\ldots,n_c\}$ is the set of
the maximal positive integers in $P$, then $P$ consists collectively of all
divisors of $n_1,\ldots,n_c$. In this case we have the following.
\begin{observation}
\label{obs:lcm-binom}
A polynomial $\psi_P(x)$ as in Definition~\ref{def:psi} where
$P$, when viewed as a poset w.r.t.~divisibility, has the set of
maximal elements $\max(P) = \{n_1,\ldots,n_c\}$, can be written as
\[
\psi_P(x) = \lcm(x^{n_1}-1,x^{n_2}-1,\ldots,x^{n_c}-1),
\]
the least common multiple of a collection of binomials in $R^{\pm}$.
\end{observation}
Each psi-polynomial is powered by Proposition~\ref{prp:lcm}
and we will see that these polynomials form a basis for the complex vector space
of power-polynomials in $\comps[x]$.

A nice thing about $\psi_P(x)$ in Observation~\ref{obs:lcm-binom}
is that it has an explicit form as a rational function in terms
of the defining binomials. This can be obtained by the Inclusion/Exclusion
Principle.
\begin{lemma}
\label{lmm:explicit}
For the polynomial $\psi_P$ of Observation~\ref{obs:lcm-binom},
\[
\psi_P(x) = \prod_{i=1}^c
\left(\prod_{S\subseteq\binom{[k]}{i}} (x^{\gcd(n_j : j\in S)} - 1)
\right)^{(-1)^{i-1}}.
\]
\end{lemma}
\begin{example}
\label{exa:n1n2n3}
When the index set $P$ has three maximal elements
$\max(P) = \{n_1,n_2,n_3\}$ then by Lemma~\ref{lmm:explicit}
the corresponding $\psi$ has the form
\[
\psi(x) = \frac{(x^{n_1}-1)(x^{n_2}-1)(x^{n_3}-1)(x^{\gcd(n_1,n_2,n_3)} - 1)}
{(x^{\gcd(n_1,n_2)} - 1)(x^{\gcd(n_1,n_3)} - 1)(x^{\gcd(n_2,n_3)} - 1)}.
\]
\end{example}
Let $P$ be a finite poset. For any antichain $A\subseteq P$
let $D(A) = \{x\in P : x\leq a \mbox{ for some }a\in A\}$ denote
the {\em downset of $A$}. If $P$ is such that $x\leq y\in P$ implies $x\in P$,
then $P = D(\max(P))$ is uniquely determined or generated by its
maximal elements. If $A$ and $B$ are antichains in $P$ then 
$D(A) = D(B)$ if and only if $A = B$. This yields a partial order
on the set ${\cal{A}}(P)$ of antichains of any finite poset $P$.
\begin{definition}
\label{def:antichain-poset}
Given a poset $P$, let ${\cal A}(P)$ denote the set of antichains in $P$.
Let $\preceq$ denote the partial order on ${\cal A}(P)$ in which, for
$A, B \in {\cal A}(P)$, 
\[
A\preceq B \Leftrightarrow D(A) \subseteq D(B).
\]
\end{definition}
With the notion of downsets we can generalize Definition~\ref{def:psi}  
of the psi-polynomials by setting $\psi_C(x) := \psi_{D(C)}(x)$
for any finite subset $C\subseteq\nats$ where $\nats$ is the poset
w.r.t.~divisibility. In particular, if $A$ is a finite antichain of $\nats$,
then, by Observation~\ref{obs:lcm-binom} we have
\begin{equation}
\label{eqn:lcm-A}  
\psi_A(x) = \lcm(x^a-1 : a\in A).
\end{equation}
By Lemma~\ref{lmm:explicit} $\psi_A(x)$ can be written as a
rational function of binomials.

The following provides a second characterization of power-polynomials.
\begin{theorem}
\label{thm:2nd-char}
A polynomial $f\in R^{\pm}$ is a power-polynomial if and only if $f$ is
a product of psi-polynomials as in (\ref{eqn:lcm-A}):
\begin{equation}
\label{eqn:psi-prod}
f(x) = \prod_{i=1}^h \psi_{A_i}(x)^{\alpha_i},
\end{equation}
for some collection of finite antichains $A_1,\ldots,A_h$ of $\nats$
w.r.t.~divisibility.
The product in (\ref{eqn:psi-prod}) can be chosen
in a nested way:  $A_1\prec A_2 \prec \cdots \prec A_h$, in which
case such a representation
of $f(x)$ is unique.
\end{theorem}
\begin{proof}
That the product in (\ref{eqn:psi-prod}) can be chosen in a nested way
where $A_1\prec A_2 \prec \cdots \prec A_h$ follows directly from
(\ref{eqn:psi}) and Definition~\ref{def:psi}.

If $A_1\prec A_2 \prec \cdots \prec A_h$ is a nested sequence
of antichains of $\nats$, then the corresponding psi-polynomials 
also form a nested sequence 
$\psi_{A_1}(x) | \psi_{A_2}(x) | \cdots | \psi_{A_h}(x)$
w.r.t.~divisibility in the polynomial ring $\ints[x]$. That the product
in (\ref{eqn:psi-prod}) is unique when it is in a nested
way then follows by induction on $h$; the number of antichains in the product
and the unique factorization of the polynomial $f(x)$ in (\ref{eqn:psi-prod})
into its irreducible cyclotomic $\phi_i(x)$ factors.
\end{proof}
\begin{example}
\label{exa:psi-n1n2n3}  
Let $\{n_1,n_2,n_3\}$ be an antichain in $\nats$ w.r.t.~divisibility,
so that none divides another.
By definition of the psi-polynomial we then have
\begin{eqnarray*}
&   & \psi_{\{n_1,n_2\}}(x)\psi_{\{n_3\}}(x)\psi_{\{\gcd(n_1,n_2)\}}(x) \\
& = & \psi_{\{n_1,n_3\}}(x)\psi_{\{n_2\}}(x)\psi_{\{\gcd(n_1,n_3)\}}(x) \\
& = & \psi_{\{n_2,n_3\}}(x)\psi_{\{n_1\}}(x)\psi_{\{\gcd(n_2,n_3)\}}(x) \\
& = & \psi_{\{n_1,n_2,n_3\}}(x)\psi_{\{\gcd(n_1,n_2),\gcd(n_1,n_3),\gcd(n_2,n_3)\}}(x)
\psi_{\{\gcd(n_1,n_2,n_3)\}}(x),
\end{eqnarray*}
the last expression is unique since
\[
\{\gcd(n_1,n_2,n_3)\}
  \prec \{\gcd(n_1,n_2),\gcd(n_1,n_3),\gcd(n_2,n_3)\}
  \prec \{n_1,n_2,n_3\}
\]
is a nested sequence of antichains.
\end{example}

\subsection{Generation and computation of power-polynomials}
\label{subsec:generation}

Here we discuss some further properties of power-polynomials.
In particular, we take a slightly more systematic approach and discuss
how a given polynomial in one variable in either $R= {\comps}[x]$
or $R^{\pm} = {\comps}[x]^{\pm}$ can {\em generate} a power-polynomial,
both from below and from above.

To add to the discussion we started with Proposition~\ref{prp:lcm},
consider $f,g\in R\subseteq R^{\pm}$. If $d = \gcd(f,g)$, then since
$d$ divides both $f$ and $g$, we have $d^{(i)}$ divides both $f^{(i)}$
and $g^{(i)}$ and so $d^{(i)}$ divides $\gcd(f^{(i)},g^{(i)})$. As $R$ and
$R^{\pm}$ are PID, $d$ is a linear combination of $f$ and $g$, say
$d = af + bg$, and hence $d^{(i)} = a^{(i)}f^{(i)} + b^{(i)}g^{(i)}$ for each
$i\in{\nats}$. This implies that $\gcd(f^{(i)},g^{(i)})$ divides $d^{(i)}$ and
so $\gcd(f,g)^{(i)} = \gcd(f^{(i)}, g^{(i)})$ for each $i\in\nats$. From this
we then obtain
\[
\lcm(f,g)^{(i)} = \left(\frac{fg}{\gcd(f,g)}\right)^{(i)}
= \frac{(fg)^{(i)}}{\gcd(f,g)^{(i)}}
= \frac{f^{(i)}g^{(i)}}{\gcd(f^{(i)},g^{(i)})} = \lcm(f^{(i)},g^{(i)}),
\]
for each $i\in\nats$ as well, and so we have the following.
\begin{observation}
\label{obs:gcd-lcm-i}
For every $f,g\in {\comps}[x]\subseteq {\comps}[x]^{\pm}$ and every $i\in\nats$
we have $\gcd(f,g)^{(i)} = \gcd(f^{(i)}, g^{(i)})$ and
$\lcm(f,g)^{(i)} = \lcm(f^{(i)}, g^{(i)})$.
\end{observation}
If $I$ is an ideal of $R = {\comps}[x]$ or $R^{\pm}$, then $I = (f)$ is
principal and hence also its power-closure
$I^{(*)} = (f^{(*)})$ for some $f^{*}\in R$ or $R^{\pm}$.
Since $\lambda(f) \leq \deg(f)+1$
we have by Proposition~\ref{prp:*-gen} the following.
\begin{observation}
\label{obs:f*}  
For any $f \in {\comps}[x]\subseteq {\comps}[x]^{\pm}$ we have
\[
f^{(*)} = \gcd(f^{(i)} : i\in{\nats}) = \gcd(f,f^{(2)},\ldots,f^{(\deg(f)+1)}).
\]
\end{observation}
Note that $f$ is powered if and only if $f^{(*)} = f$. Hence, directly from
the above observation we get the following analogous to
Proposition~\ref{prp:lcm}.
\begin{proposition}
\label{prp:gcd}
For any $f, g \in {\comps}[x]\subseteq {\comps}[x]^{\pm}$ we have
$\gcd(f,g)^{(*)} = \gcd(f^{(*)},g^{(*)})$. In particular, if
both $f$ and $g$ are powered, then $\gcd(f,g)$ is also powered.
\end{proposition}
We note that for a polynomial $c$, 
(i) if $f$ and $g$ both divide $c$ then so does $\lcm(f,g)$.
(ii) if both $f$ and $g$ are both divisible by $c$, then so
is $\gcd(f,g)$. With this in mind we obtain 
the following corollary from Propositions~\ref{prp:lcm} and~\ref{prp:gcd}.
\begin{corollary}
\label{cor:lattice}
Monic power-polynomials in ${\comps}[x]\subseteq {\comps}[x]^{\pm}$
form a lattice w.r.t.~divisibility, where the meet is $\gcd$ and the
join is $\lcm$.

Further, for any polynomials $a$ and $b$   
(i) the downset $D(b) = \{f : f|b\}$, and
(ii) the upset $U(a) = \{f : a|f\}$, and
(iii) the interval $[a,b] = \{f : a|f|b \}$ are all lattices.
\end{corollary}

We now discuss a description and the computation of
$f^{(*)}$ for a given polynomial $f$
in ${\comps}[x]$ or ${\comps}[x]^{\pm}$. Suppose $g$ is powered and divides $f$.
For each $i\in\nats$ we then have that $g | g^{(i)} | f^{(i)}$ and hence
$g$ divides $\gcd(f^{(i)} : i\in\nats) = f^{(*)}$. Therefore we note
the following.
\begin{claim}
\label{clm:max}
$f^{(*)}$ is the unique maximum polynomial $g$ w.r.t.~divisibility that
satisfies (i) $g|f$ and (ii) $g$ is powered.
\end{claim}
{\sc Remark:} Needless to say, the existence of such a maximal polynomial as
described in Claim~\ref{clm:max} is a direct consequence of
Corollary~\ref{cor:lattice}. However, that $D(f)$ is a lattice is not needed
to prove Claim~\ref{clm:max} as was done.

If $f(x) = x^nc(x)$ where $c(0)\neq 0$, then by Claim~\ref{clm:max} or by
Observation~\ref{obs:f*} we clearly have
$f^{(*)} = x^nc^{(*)}$ and so $f$ is powered iff $c$ is powered.
We will therefore, unless otherwise stated,
from now on in this subsection concentrate on
$f\in {\comps}[x]^{\pm}$ with no zero root. In this case,
since $f^{(*)}$ is powered, we have by Observation~\ref{obs:roots-of-unity}
that each root of $f^{(*)}$ is root of unity and therefore has the
form described in Proposition~\ref{prp:1st-char}. By Claim~\ref{clm:max}
we then have that  
\begin{equation}
\label{eqn:Phi}  
f^{(*)}(x) = \prod_{n\geq 1}\phi_n(x)^{p_n}
\end{equation}
where the exponents 
$p_n$ are uniquely determined by Claim~\ref{clm:max} and can be
computed as follows: For a given $f\in {\comps}[x]^{\pm}$ we can write
$f = \Phi_f r_f$ where $\Phi_f = \prod_{n\geq 1}\phi_n^{k_n}$ and $r_f(x)$ is not
divisible by any cyclotomic polynomial. Note that $f^{(*)} = \Phi_f^{(*)}$.
By Claim~\ref{clm:max} and Proposition~\ref{prp:1st-char} we have the
following.
\begin{proposition}
\label{prp:pn}
If $f\in {\comps}[x]^{\pm}$ and $\Phi_f = \prod_{n\geq 1}\phi_n^{k_n}$, then
the exponents $p_n$ from (\ref{eqn:Phi}) can be computed
by $p_n := \min(\{k_d : d|n\})$.
\end{proposition}
\begin{example}
\label{exa:f*}  
Consider the polynomial $f$ given by
\[
f(x) = \phi_{12}(x)^6\phi_{8}(x)^3\phi_{6}(x)^5\phi_{4}(x)^4
\phi_{3}(x)^2\phi_{2}(x)^3\phi_{1}(x)^4.
\]
Here $f = \Phi_f$, so by Proposition~\ref{prp:pn} we obtain
\[
f^{(*)}(x) = \phi_{12}(x)^2\phi_{8}(x)^3\phi_{6}(x)^2\phi_{4}(x)^3
\phi_{3}(x)^2\phi_{2}(x)^3\phi_{1}(x)^4.
\]
\end{example}
If $I$ and $J$ are ideals of ${\comps}[x]^{\pm}$, a PID, then
$I = (f)$ and $J = (g)$ for some polynomials $f$ and $g$. In this
case we have
\[
(I\cap J)^{(*)} = ((f)\cap (g))^{(*)} = (\lcm(f,g))^{(*)}
= (\lcm(f,g)^{(*)}),
\]
and 
\[
I^{(*)}\cap J^{(*)} = (f^{(*)})\cap (g^{(*)}) = (\lcm(f^{(*)},g^{(*)}).
\]
With this in mind, Proposition~\ref{prp:pn} provides a recipe for
examples of ideals $I$ and $J$ of 
${\comps}[x]^{\pm}$ such that the containment
$(I\cap J)^{(*)}\subset I^{(*)}\cap J^{(*)}$
is strict, namely whenever we have polynomials $f,g\in {\comps}[x]^{\pm}$
neither of which are powered, but where their least common multiple
$\lcm(f,g)$ is powered.
\begin{example}
\label{exa:42121}  
Consider the polynomials 
$f = \phi_4^2\phi_2\phi_1^2$ and $g = \phi_2^2\phi_1$ from ${\comps}[x]^{\pm}$.
By Proposition~\ref{prp:1st-char} neither $f$ nor $g$ are powered. However,
$\lcm(f,g) = \phi_4^2\phi_2^2\phi_1^2$ is powered and hence we obtain here
$(\lcm(f,g))^{(*)} = (\lcm(f,g)) = (\phi_4^2\phi_2^2\phi_1^2)$. However,
by Proposition~\ref{prp:pn} we obtain $f^{(*)} = \phi_4\phi_2\phi_1^2$
and $g^{(*)} = \phi_2\phi_1$ and so
$(f)^{(*)}\cap (g)^{(*)} = (\lcm(f^{(*)},g^{(*)})) = (\phi_4\phi_2\phi_1^2)$,
and so the containment
$(I\cap J)^{(*)}\subset I^{(*)}\cap J^{(*)}$ is indeed strict. Such
examples then complement the ones given
in Examples~\ref{exa:yax-ybx} and~\ref{exa:zaxby-cont}
as they demonstrate that even in the case for PIDs we cannot force equality.
\end{example}
As with the power-closure, if $I$ is an ideal of
$R = {\comps}[x]$ or $R^{\pm} = {\comps}[x]^{\pm}$, then $I = (f)$ is
principal and hence also dually its power-interior 
$I^{(\circ)} = (f^{(\circ)})$ for some $f^{(o)}$.
By the definition of $f^{(*)}$ here above
and Observation~\ref{obs:close-int} we have that
\[
(f^{(\circ)(*)}) = (f^{(\circ)})^{(*)}= ((f)^{(\circ)})^{(*)}
= (f)^{(\circ)} = (f^{(\circ)}),
\]
and so $f^{(\circ)}$ is powered. By Observation~\ref{obs:MC-below}
we have $(f^{(\circ)}) = \{ g : g^{(i)}\in (f) \mbox{ for each }i\in{\nats}\}$
and so by Observation~\ref{obs:f*} that $(f^{(\circ)}) = (g : f|g^{(*)})$.
From this we get the following straightforward yet notable equivalence, namely
for polynomials $f$ and $g$ in $R$ or $R^{\pm}$ we have
\begin{equation}
\label{eqn:fg}
f | g^{(*)} \Leftrightarrow f^{(\circ)} | g .
\end{equation}
For $g = f^{(\circ)}$ we get from (\ref{eqn:fg}) that $f | (f^{(\circ)})^{(*)}$
and so $f | f^{(\circ)}$, since $f^{(\circ)}$ is powered. On the other hand if
$f | g$ where $g$ is powered, then $f | g^{(*)}$ and so by (\ref{eqn:fg})
$f^{(\circ)} | g$. This yields the following description of $f^{(\circ)}$
which is analogous to that of $f^{(*)}$ in Claim~\ref{clm:max}.
\begin{claim}
\label{clm:min}
$f^{(\circ)}$ is the unique minimum polynomial $g$ w.r.t.~divisibility 
that satisfies (i) $f|g$ and (ii) $g$ is powered.
\end{claim}
{\sc Remark:} As for $f^{(*)}$, the existence of such a minimal polynomial as
described in Claim~\ref{clm:min} is a direct consequence of
Corollary~\ref{cor:lattice}. However, that $U(f)$ is a lattice is not needed
to prove Claim~\ref{clm:min}.

Again, if $f(x) = x^nc(x)$ where $c(0)\neq 0$, then by Claim~\ref{clm:min}
we clearly have $f^{(\circ)} = x^nc^{(\circ)}$ and so, in particular, $f$ is
powered iff $c$ is powered. This justifies our continuing assumption that our
polynomials are from ${\comps}[x]^{\pm}$ and therefore with no zero root.
In this case, if $f|g$ where $g$ is powered, then
by Observation~\ref{obs:roots-of-unity} either $g = 0$ or
each root of $g$ is root of unity and therefore has the
form described in Proposition~\ref{prp:1st-char}. In particular,
if $f$ has roots that are not roots of unity, then $g = 0$ must hold.
Otherwise, we can assume each root of $f$ to be a root of unity, in which
case $f$ has the form described in Observation~\ref{obs:roots-of-unity},
namely $f(x) = (x-\rho_1)^{a_1}(x-\rho_2)^{a_2}\cdots(x-\rho_k)^{a_k}$.
To compute $f^{(\circ)}$ as described in Claim~\ref{clm:min}, we see
that $f^{(\circ)}$ must have the form described in
Proposition~\ref{prp:1st-char}. For each $n\in\nats$ let ${\ell}_n$
be the largest
multiplicity of a primitive $n$-root of unity in $f$. In this case $f^{(\circ)}$
must be divisible by $\Lambda_f$ given by
\begin{equation}
\label{eqn:Lf}  
\Lambda_f = \prod_{n\geq 1}\phi_n^{{\ell}_n}.
\end{equation}
Note that $f|\Lambda_f$ and $f^{(\circ)}$ is the minimum polynomial
w.r.t.~divisibility that is powered and is divisible by $\Lambda_f$, that is
to say $f^{(\circ)} = \Lambda_f^{(\circ)}$.
By Claim~\ref{clm:max} and Proposition~\ref{prp:1st-char} we have the
following.
\begin{proposition}
\label{prp:qn}
If $f\in {\comps}[x]^{\pm}$ and $\Lambda_f$ is as in (\ref{eqn:Lf}),
then $f^{(\circ)} = \prod_{n\geq 1}\phi_n^{q_n}$ where the exponents
$q_n$ are computed by $q_n := \max(\{{\ell}_m : n|m \})$.
\end{proposition}
\begin{example}
\label{exa:fo}
Consider again the polynomial $f$ from Example~\ref{exa:f*}
given by
\[
f(x) = \phi_{12}(x)^6\phi_{8}(x)^3\phi_{6}(x)^5\phi_{4}(x)^4
\phi_{3}(x)^2\phi_{2}(x)^3\phi_{1}(x)^4.
\]
Here $f = \Lambda_f$, so by Proposition~\ref{prp:qn} we obtain
\[
f^{(\circ)}(x) = \phi_{12}(x)^6\phi_{8}(x)^3\phi_{6}(x)^6\phi_{4}(x)^6
\phi_{3}(x)^6\phi_{2}(x)^6\phi_{1}(x)^6.
\]
\end{example}
If $I$ and $J$ are ideals of ${\comps}[x]^{\pm}$, a PID, then
$I = (f)$ and $J = (g)$ for some polynomials $f$ and $g$. In this
case we have
\[
(I + J)^{(\circ)} = ((f) + (g))^{(\circ)} = (\gcd(f,g))^{(\circ)}
= (\gcd(f,g)^{(\circ)}),
\]
and 
\[
I^{(\circ)} + J^{(\circ)} = (f^{(\circ)}) + (g^{(\circ)})
= (\gcd(f^{(\circ)},g^{(\circ)}).
\]
With this in mind, Proposition~\ref{prp:qn} provides a recipe for
examples of ideals $I$ and $J$ of 
${\comps}[x]^{\pm}$ such that the containment
$I^{(\circ)} + J^{(\circ)}\subset(I + J)^{(\circ)}$
is strict, namely whenever we have polynomials $f,g\in {\comps}[x]^{\pm}$
neither of which are powered, but where their greatest common denominator 
$\gcd(f,g)$ is powered.
\begin{example}
\label{exa:42121o}
Consider the polynomials 
$f = \phi_4^2\phi_2\phi_1^2$ and $g = \phi_2^2\phi_1$ from ${\comps}[x]^{\pm}$.
By Proposition~\ref{prp:1st-char} neither $f$ nor $g$ are powered. However,
$\gcd(f,g) = \phi_2\phi_1$ is powered and hence we obtain here
$(\gcd(f,g))^{(\circ)} = (\gcd(f,g)) = (\phi_2\phi_1)$. However,
by Proposition~\ref{prp:qn} we obtain $f^{(\circ)} = \phi_4^2\phi_2^2\phi_1^2$
and $g^{(\circ)} = \phi_2^2\phi_1^2$ and so
$(f)^{(\circ)} + (g)^{(\circ)} = (\gcd(f^{(\circ)},g^{(\circ)}))
= (\phi_2^2\phi_1^2)$, and so for $I = (f)$ and $J = (g)$
we get that the containment 
$I^{(\circ)} + J^{(\circ)}\subset (I + J)^{(\circ)}$ is indeed strict. Such
examples then complement the trivial one given in
Example~\ref{exa:IplusJ-circ} and provide a complete class of such
examples for PIDs.
\end{example}
{\sc Remark:} For a given polynomial $f(x)$ both the polynomials
$f^{(*)}$ and $f^{(\circ)}$ are powered, and can therefore readily be put
into the form consisting of products of $\psi_{A_i}(x)$
as stated in Theorem~\ref{thm:2nd-char}. The antichains $A_1,\ldots,A_h$
in ${\nats}$ are contained in a finite poset determined by
$f(x)\in {\comps}[x]^{\pm}$. The following yields a bound on how large
this poset really is.
\begin{corollary}
\label{cor:psi-downset}
For a given polynomial $f(x)\in{\comps}[x]^{\pm}$
when writing the power-polynomials $f^{(*)}$ and
$f^{(\circ)}$ in terms of $\psi_{A_i}(x)$ as in
Theorem~\ref{thm:2nd-char}, then the nested antichains
$A_1\prec A_2 \prec \cdots \prec A_h$, and their corresponding downsets,
are all contained in the downset $D(A_h)$.

For $f^{(*)}$ this downset is given by 
$D(A_h) = D^{(*)}(A_h) = \bigcup_{n\in P^{(*)}(f)}D(n)\subseteq {\nats}$
w.r.t.~divisibility, where $P^{(*)}(f)\subseteq {\nats}$ is the finite set
of positive integers $n\in\nats$ such that $f$ is divisible by the cyclotomic
polynomials $\phi_d$ for each $d|n$.

For $f^{(\circ)}$ this downset is given by 
$D(A_h) = D^{(\circ)}(A_h) = \bigcup_{n\in P^{(\circ)}(f)}D(n)\subseteq {\nats}$
w.r.t.~divisibility, where $P^{(\circ)}(f)\subseteq {\nats}$ is the finite set
of positive integers $n\in\nats$ such that $f$ contains a primitive $n$-th
root of unity.
\end{corollary}
\begin{proof}
That the nested antichains and their corresponding downsets are all
contained in the downset $D(A_h)$ of $A_h$ is clear by the mere definition
of the partial order $\preceq$ from Definition~\ref{def:antichain-poset}.

That $D(A_h) = D^{(*)}(A_h)$ and $D(A_h) = D^{(\circ)}(A_h)$ have the
described form follows from Propositions~\ref{prp:pn} and~\ref{prp:qn}.
\end{proof}

\section{Applications to general principal ideals}
\label{sec:gen-principal}

In this section we will use results in the previous section
to characterize all principal power-closed ideals in 
the usual polynomial ring $R = \comps[x_1,\ldots,x_d]$
as well as in the corresponding Laurent polynomial ring
$R^{\pm}= \comps[x_1,\ldots,x_d]^{\pm}$. For a $d$-tuple $\tilde{p}\in {\ints}^d$
let $\supp(\tilde{p}) = \{ i\in [d] : p_i\neq 0\}$ denote
the {\em support} of $\tilde{p}$. For an arbitrary collection of tuples
$\tilde{p}_1,\ldots,\tilde{p}_k$ let $\gcd(\tilde{p}_1,\ldots,\tilde{p}_k)$
denote the greatest common divisor of all the coordinates collectively:
$\gcd(\tilde{p}_1,\ldots,\tilde{p}_k) 
= \gcd(\{ p_{i\/j} : 1\leq i\leq k, \ \ 1\leq j\leq d\})$.
\begin{definition}
\label{def:disjointed}
Call a pair of $d$-tuples $\tilde{p},\tilde{q}\in {\ints}^d$ 
{\em primitive} if 
(i) $\supp(\tilde{p})\cap\supp(\tilde{q})=\emptyset$ and 
(ii) $\gcd(\tilde{p},\tilde{q}) = 1$.

Call a binomial $\tilde{x}^{\tilde{p}} - \tilde{x}^{\tilde{q}} \in R^{\pm}$
{\em power primitive} if $\tilde{p}$ and $\tilde{q}$ 
are primitive. 
\end{definition}
Consider an arbitrary and fixed power-primitive binomial $b\in R^{\pm}$.
By renaming the variables $x_1,\ldots,x_d$ we can assume
that $b$ has the form 
$b = x_1^{p_1}\cdots x_{\ell}^{p_h} - y_1^{q_1}\cdots y_k^{q_k}$
where $\gcd(\tilde{p},\tilde{q}) = 1$ and where all the coordinates
$p_i, q_j$ are positive integers.
Note that 
\[
b = \upsilon
\left(\frac{y_1^{q_1}\cdots y_k^{q_k}}{x_1^{p_1}\cdots x_h^{p_h}} - 1\right)
\]
for some unit $\upsilon \in R^{\pm}$. Clearly $b$ is irreducible
in $R^{\pm}$ if and only if 
$\frac{y_1^{q_1}\cdots y_k^{q_k}}{x_1^{p_1}\cdots x_h^{p_h}} - 1$ is irreducible
in $R^{\pm}$. The valuation 
\[
\nu(b) := 
\sum_{i=1}^hp_i + \sum_{i=1}^kq_i \in \nats
\]
is well defined for all binomials in $R^{\pm}$
of the form $\frac{y_1^{q_1}\cdots y_k^{q_k}}{x_1^{p_1}\cdots x_h^{p_h}} - 1$.
\begin{lemma}
\label{lmm:irreducible-d}
A binomial $b = x_1^{p_1}\cdots x_{\ell}^{p_h} - y_1^{q_1}\cdots y_k^{q_k} \in R$
is irreducible in both $R$ and $R^{\pm}$ if and only if
$\gcd(\tilde{p},\tilde{q}) = 1$.
\end{lemma}
\begin{proof}
It is evident that the condition $\gcd(\tilde{p},\tilde{q}) = 1$ is necessary
for $b$ to be irreducible.
Also, it is clear that every binomial $b$ with $\nu(b) = 1$ is irreducible.

If $q_1\geq p_1$ then in terms of new variables $x_1',x_2,\ldots,x_h$
and $y_1,\ldots,y_k$, where $x_1' := x_1y_1$, we get that $b$ 
is irreducible in $R^{\pm}$ if and only if 
\[
b' 
= \frac{y_1^{q_1}\cdots y_k^{q_k}}{(x_1y_1)^{p_1}\cdots x_h^{p_h}} - 1
= \frac{y_1^{q_1-p_1}\cdots y_k^{q_k}}{x_1^{p_1}\cdots x_h^{p_h}} - 1
\]
is irreducible in $R^{\pm}$. Since 
$\gcd(\tilde{p},q_1-p_1,q_2,\ldots,q_k) = \gcd(\tilde{p},\tilde{q}) = 1$
and $\nu(b') = \nu(b) - p_1 < \nu(b)$, it follows by induction 
on $\nu(b)$ that $b$ is irreducible in $R^{\pm}$ and hence also
in $R$.

The case when $p_1\geq q_1$ is treated in a similar fashion.
\end{proof}
By changing one of the variables $x_1,\ldots,x_h$, or one of the
variables $y_1,\ldots,y_k$, to a complex scalar multiple of itself
we get the following corollary.
\begin{corollary}
\label{cor:irreducible-d}
If $c\in\comps\setminus\{0\}$, then 
a polynomial $f = \tilde{x}^{\tilde{q}} - c\tilde{x}^{\tilde{p}}$
is irreducible in $R^{\pm}$ if and only if $\tilde{p}$ and $\tilde{q}$
form a primitive pair of $d$-tuples.
\end{corollary}
Note that, up to a unit in $R^{\pm}$ the polynomial  $f$ of the corollary
is an associate of 
$\tilde{\xi} - c$ where
$\tilde{\xi} = \tilde{x}^{\tilde{q}}/\tilde{x}^{\tilde{p}} \in R^{\pm}$.
\begin{lemma}
\label{lmm:ass}
Two irreducibles $\tilde{\xi} - c$ and $\tilde{\xi}' - c'$
with $c,c'\in {\comps}\setminus \{0\}$ 
are associates in $R^{\pm}$ if and only if
$(\tilde{\xi}',c') = (\tilde{\xi},c)$ or
$(\tilde{\xi}',c') = (\tilde{\xi}^{-1},c^{-1})$.
\end{lemma}
\begin{proof}
Assume $\tilde{\xi} - c$ and $\tilde{\xi}' - c'$ are associates.
In that case there is a unit $\gamma \in {\comps}^*[x_1,\ldots,x_d]$
with $\tilde{\xi}' - c' = \gamma(\tilde{\xi} - c)$. Since $c'\in{\comps}^*$
then we have either $c' = \gamma c$ or $c' = -\gamma\tilde{\xi}$.

If $c' = \gamma c$, the we get
$\tilde{\xi}' = \gamma\tilde{\xi}$ and hence $\gamma = 1$, from
which we conclude that $(\tilde{\xi}',c') = (\tilde{\xi},c)$.

If $c' = -\gamma\tilde{\xi}$, then 
we have $\gamma = {-c'}\tilde{\xi}^{-1}$ and so
$\tilde{\xi}' - c' = {-c'}\tilde{\xi}^{-1}(\tilde{\xi} - c)
= -c' + c'c\tilde{\xi}^{-1}$ or $\tilde{\xi}' = c'c\tilde{\xi}^{-1}$.
From this last equation we conclude that $c'c =1$ and
$\tilde{\xi}' = \tilde{\xi}^{-1}$ and therefore
$(\tilde{\xi}',c') = (\tilde{\xi}^{-1},c^{-1})$.
\end{proof}    

Each arbitrary binomial $b = \tilde{x}^{\tilde{q}} - \tilde{x}^{\tilde{p}}$ 
in $R$ can be written as
$b = \tilde{x}^{\tilde{c}}(\tilde{x}^{\tilde{q}'} - \tilde{x}^{\tilde{p}'})$
where $\supp(\tilde{p}')\cap\supp(\tilde{q}') = \emptyset$.
Further, if $\gcd(\tilde{p}',\tilde{q}') = h\in\nats$, then
$\tilde{p}' = h\tilde{p}''$ and $\tilde{q}' = h\tilde{q}''$ 
where $\gcd(\tilde{p}'',\tilde{q}'') = 1$. Hence, we can write
an arbitrary binomial $b$ as 
\begin{equation}
\label{eqn:binom-factorization}
b = \tilde{x}^{\tilde{q}} - \tilde{x}^{\tilde{p}}
= \tilde{x}^{\tilde{c}}(\tilde{x}^{\tilde{q}'} - \tilde{x}^{\tilde{p}'})
= \tilde{x}^{\tilde{c}}((\tilde{x}^{\tilde{q}''})^h 
- (\tilde{x}^{\tilde{p}''})^h)
= \tilde{x}^{\tilde{c}}\prod_{i=0}^{h-1}
(\tilde{x}^{\tilde{q}''} - \rho^i\tilde{x}^{\tilde{q}''})
\end{equation}
where $\tilde{p}'', \tilde{q}''$ forms a primitive pair of $d$-tuples
and so each $\tilde{x}^{\tilde{q}''} - \rho^i\tilde{x}^{\tilde{q}''}$ 
is an irreducible polynomial by Corollary~\ref{cor:irreducible-d}.
Hence, any binomial has a factorization into irreducible polynomials
as given in (\ref{eqn:binom-factorization}).

Consider now a principal power-closed ideal $I = (f) \subseteq R^{\pm}$.
By Proposition~\ref{prp:principal} $f$ must divide
the polynomial $(\tilde{x}^{\tilde{p}_k} - \tilde{x}^{\tilde{p}_1})
  (\tilde{x}^{\tilde{p}_k} - \tilde{x}^{\tilde{p}_2})\cdots
(\tilde{x}^{\tilde{p}_k} - \tilde{x}^{\tilde{p}_{k-1}})$
and so must have, up to a unit, the form
$f = \prod_{i=1}^m(\tilde{x}^{\tilde{q}_i} - \rho_i\tilde{x}^{\tilde{p}_i})$
for some roots of unity $\rho_i$ where each pair
where $\tilde{p}_i, \tilde{q}_i$ forms a primitive pair of $d$-tuples.
For convenience we fix a term order on $[x_1,\ldots,x_d]$
and assume the binomial factors are such that
$\tilde{x}^{\tilde{q}_i} > \tilde{x}^{\tilde{p}_i}$ for each $i$.
In this case $f$ is an associate of $\prod_{i=1}^m(\tilde{\xi}_i - \rho_i)$
where two irreducibles $\tilde{\xi}_i - \rho_i$ and $\tilde{\xi}_j - \rho_j$
are by Lemma~\ref{lmm:ass} associated if and only if
$\tilde{\xi}_i = \tilde{\xi}_j$ and $\rho_i = \rho_j$. Call the rational
group element
$\tilde{\xi} = \tilde{x}^{\tilde{q}}/\tilde{x}^{\tilde{p}} \in R^{\pm}$
{\em positive} w.r.t.~our term order of $[x_1,\ldots,x_d]$
if $\tilde{x}^{\tilde{q}} > \tilde{x}^{\tilde{p}}$. We can therefore
assume $f$ to have the form $f = \prod_{i=1}^m(\tilde{\xi}_i - \rho_i)$
where $\tilde{\xi}_i$ is a positive fraction of a primitive pair
of monomials for each  $i$.

Since $R = \comps[x_1,\ldots,x_d]$ and $R^{\pm} = \comps[x_1,\ldots,x_d]^{\pm}$
are UFD for every $d\in{\nats}$, Definition~\ref{def:Minks-poly} makes sense
for each $d\in{\nats}$ as well.

Since both $R = \comps[x_1,\ldots,x_d]$ or and
$R^{\pm} = \comps[x_1,\ldots,x_d]^{\pm}$ are UFDs, then
Definition~\ref{def:Minks-poly} for the single variable case can be
generalized as the following.
\begin{definition}
\label{def:Minks-poly-d}
Let $d\in{\nats}$.
A polynomial $f = f(\tilde{x})$ in $R = \comps[x_1,\ldots,x_d]$ or
$R^{\pm} = \comps[x_1,\ldots,x_d]^{\pm}$ is {\em powered}
if $f(\tilde{x})$ divides $f^{(i)}(\tilde{x})$ for each $i\in\nats$.
A powered polynomial will be called a {\em power-polynomial}. 
Equivalently, $f$ is a power-polynomial if and only if the
principal ideal $(f)$ in either $R$ or $R^{\pm}$ is a power-closed ideal.  
\end{definition}
For a principal power-closed ideal $I = (f) \subseteq R^{\pm}$, where
the polynomial $f$ is a power-polynomial, we can, by the
above discussion, assume $f$ to have the form
\begin{equation}
\label{eqn:MP-form}
f = \prod_{i=1}^k\left(\prod_{j=1}^{\ell_i}(\tilde{\xi}_i - \rho_{i\/j})\right),
\end{equation}
where the $\tilde{\xi}_i$ are distinct positive fractions of primitive pairs
of monomials and each $\rho_{i\/j}$ is a root of unity
in ${\comps}^*$. By our assumption on positivity of the $\tilde{\xi}_i$s, none
of the elements $\tilde{\xi}_i - \rho_{i\/j}$ are associates. Since
$R^{\pm}$ is a UFD the following claim is evident about the irreducible
factors of $f$ in (\ref{eqn:MP-form}).
\begin{claim}
\label{clm:UFD-ij}
Let $n\in\nats$. The irreducible $\tilde{\xi}_i - \rho_i$,
where $\tilde{\xi}_i$ is positive fraction of primitive pairs of monomials,
divides $\tilde{\xi}_j^n - \rho_j$
if and only if $\tilde{\xi}_i = \tilde{\xi}_j$ and $\rho_i^n = \rho_j$.
\end{claim}
For each $i\in\{1,\ldots,k\}$ let
$f_i = \prod_{j=1}^{\ell_i}(\tilde{\xi}_i - \rho_{\/j})$ from (\ref{eqn:MP-form}),
so $f = f_1\cdots f_k$. By Claim~\ref{clm:UFD-ij} we have that
$f$ is a power-polynomial if and only if each $f_i$ is a
power-polynomial. Consider a fixed
element $f_i$. Since $\tilde{\xi}_i$ is a fraction of a primitive
pair of monomials there is an evaluation $x_i := t^{a_i}$ for each $i$
where $t \in \comps[t]^{\pm}$ is indeterminate, such that
$\tilde{\xi}_i = \tilde{x}^{\tilde{q}_i}/\tilde{x}^{\tilde{p}_i}$
evaluates to $t$. In this case $f_i = f_i(\tilde{\xi}_i)$ evaluates
to the polynomial $f(t)\in\comps[t]$, which must then be a
power-polynomial as well.
We therefore have the following summarizing theorem on principal
power-closed ideals of $R^{\pm}$ and $R$. 
\begin{theorem}
\label{thm:MP-d}
A principal ideal $(f)\subseteq R^{\pm} = \comps[x_1,\ldots,x_d]^{\pm}$
is a power-closed ideal if and only if
\[
f = \prod_{i=1}^k f_i(\tilde{\xi}_i),
\]
where each $f_i\in\comps[x]$ is a power-polynomial and therefore has the form
described in Proposition~\ref{prp:1st-char} and
Theorem~\ref{thm:2nd-char}, and where the $\tilde{\xi}_i$ are distinct
positive fractions of primitive pairs of monomials, for
some term order on $[x_1,\ldots,x_d]$.
\end{theorem}
Since $R = \comps[x_1,\ldots,x_d]$ is a UFD, a principal ideal
$(f)\subseteq R = \comps[x_1,\ldots,x_d]$, where $f$ is not divisible
by any irreducible monomial $x_i$, is power-closed iff
$(\tilde{x}^{\tilde{p}}f)$ is power-closed for any monomial
${\tilde{x}}^{\tilde{p}}\in R$. Passing to $R^{\pm}$ the ideal
$(f)\subseteq R \subseteq R^{\pm} = \comps[x_1,\ldots,x_d]^{\pm}$ is
power-closed iff $f$ has the form given in the above Theorem~\ref{thm:MP-d}
up to a unit in $R^{\pm}$. Since the units in $R^{\pm}$ are exactly
scalar multiples of monomials from $S = [x_1,\ldots,x_d]$ we have
the following corollary.
\begin{corollary}
\label{cor:MP-d}
A principal ideal $(f)\subseteq R = \comps[x_1,\ldots,x_d]$
is a power-closed ideal if and only if
\[
f = \tilde{x}^{\tilde{p}}\prod_{i=1}^k f_i(\tilde{\xi}_i),
\]
where each $f_i\in\comps[x]$ is a power-polynomial and therefore has the form
described in Proposition~\ref{prp:1st-char} and
Theorem~\ref{thm:2nd-char}, and where the $\tilde{\xi}_i$ are distinct
positive fractions of primitive pairs of monomials, for
some term order on $[x_1,\ldots,x_d]$.
\end{corollary}
{\sc Remark:}
For ideals of $R$ and $R^{\pm}$ when $d\geq 2$ that are not principal, there
are plenty of power-closed ideals that are not associated with roots of
unity as we saw in Examples~\ref{exa:yax} and~\ref{exa:zaxby}.

\section{The radical}
\label{sec:rel-rad}

We have seen that the power-closure is a closure operator.
We turn now to consider a familiar operator, the
radical, $\sqrt{I}$, defined here in the usual way for ideals $I$.
It can be viewed as a closure operator on the lattice of ideals since, if
$I$ is an ideal then $\sqrt{I}$ is also an ideal, and the operator is
enlarging, monotone, and idempotent.
See particularly~\cite[p.~9]{Atiyah-Macdonald}.  

As usual, given a subset (usually, an ideal)
$I\subseteq {\comps}[x_1,\ldots,x_d]$, the zero locus, or the affine variety,
of $I$ is $Z(I) = \{\tilde{w} \in ({\comps}^*)^d : f(\tilde{w}) = 0$ for
each $f\in I\}$; and the {\em radical}, or {\em nilradical}, $\sqrt{I}$,
of an ideal $I$ is
$\{f\in {\comps}[x_1,\ldots,x_d] : f^n\in I\mbox{ for some $n\in{\nats}$ }\}$.
If $I = \sqrt{I}$ then the ideal $I$ is said to be {\em radical}.
Given a subset
$Z\subseteq ({\comps}^*)^d$, $I(Z)$ denotes the set of polynomials
$f\in {\comps}[x_1,\ldots,x_d]$ such that $f(\tilde w) = 0$ for each
$\tilde{w}\in Z$.

By Hilbert's Nullstellensatz~\cite[p.~85]{Atiyah-Macdonald},
${\comps}$ being algebraically closed, we have
$\sqrt{I^{(*)}} = I(Z(I^{(*)}))$.

The functions $I(\cdot)$ and $Z(\cdot)$ form a Galois connection
between the set of ideals in $\comps[x_1, \ldots, x_d]$ and the
collection of subsets of $\comps^d$,
each partially ordered by inclusion.  Using Hilbert's theorem it is seen
that this Galois connection yields a bijective order-reversing correspondence
between the set of radical ideals and the set of zero loci
of ideals $I$ of $\comps[x_1,\ldots, x_d]$.  We will establish a
similar correspondence between the set of radical power-closed ideals
of $\comps[x_1,\ldots,x_d]^\pm$ and certain subsets of the $d$-dimensional
torus group $T^d\subseteq (\comps^*)^d$.

Let $T$ denote the {\em circle group}, consisting of complex numbers
of unit modulus, under multiplication;  then $T^d \subseteq \comps^d$
is the {\em $d$-dimensional torus group}.
(Here we depart from the more common terminology in algebraic geometry, 
wherein the torus is considered to be complex algebraic manifold
$(\comps^*)^d$.)

If $f(\tilde x)$ is a binomial then the set
of solutions $\tilde x\in(\comps^*)^d$ to the equation $f(\tilde x) = 0$ is
a topologically closed subgroup of the group $(\comps^*)^d$
(under componentwise multiplication). The intersection of this subgroup
with $T^d$, that is, the set of solutions in $T^d$ to $f(\tilde x) = 0$, is
a topological group that is up to isomorphism the product of a
finite cyclic group with the $(d-1)$-dimensional  torus group. 
Two binomials yield the same subgroup in this way if and only if
they differ by a multiplicative factor of $\pm1$ times a monomial.
Any proper subgroup of $T^d$ can be obtained
as the intersection of such subgroups.

For $\tilde w = (w_1, \ldots, w_d)$, let $\tilde w^{(j)}$ denote
$(w_1^j, \ldots, w_d^j)$. Given an ideal $I$, let $Z_T(I)$ be the set of
$\tilde w\in T^d$ for which $f(\tilde w^{(j)})=0$ for each $f\in I$ and
$j\in \nats$. When $I$ is a power-closed ideal, $Z_T(I) = Z(I)\cap T^d$.

For $S\subseteq \comps^d$, let
$I_T(S) = \{f\in \comps[x_1,\ldots,x_d] : f(\tilde w^{(j)}) = 0$
for each $\tilde w \in S$ and $j\in \nats\}$.
Also write $I_T(\tilde w) = I_T(\{\tilde w\})$.

For $f\in\comps[x_1,\ldots,x_d]$, the {\em monomial support} of $f$ is
the set of monomials in $[x_1,\ldots,x_d]$ appearing with nonzero
coefficient in its representation as a linear combination of
distinct monomials.
\begin{lemma}
\label{lmm:mon-binom}
Suppose $\tilde w \in \comps^d$ and let $I$
be the ideal $I_T(\tilde w)$. If $f$ is a nonzero polynomial in
$I = I_T(\tilde w)$ having minimal monomial support then $f$ is a
constant multiple of a monomial or a binomial.
\end{lemma}
\begin{proof}
Let $f$ be a nonzero polynomial in $I$ having minimal nonempty support,
where, in terms of the linear basis of monomials,
$f(\tilde x) = \sum_{i=1}^k a_i\tilde x^{\tilde p_i}$.
Then $\sum_{i=1}^k a_i(\tilde w^{\tilde p_i})^j = f(\tilde w^{(j)}) = 0$
($j= 1, 2, \ldots$).
Let $V = V({\tilde{w}}^{\tilde{p}_1},\ldots,{\tilde{w}}^{\tilde{p}_k})$
denote the $k\times k$ Vandermonde matrix for which
$V_{i\/j} = (\tilde w^{\tilde p_i})^{j-1}$, and let
$D = D({\tilde{w}}^{\tilde{p}_1},\ldots,{\tilde{w}}^{\tilde{p}_k})$
denote the $k\times k$ diagonal matrix having
$D_{i\/i} = \tilde w^{\tilde p_i}$. As is immediate from the well-known
expression for the determinant of a Vandermonde matrix,
the determinant of $V$ is
$\prod_{i<j}(\tilde w^{\tilde p_j}-\tilde w^{\tilde p_i})$;
and of course the determinant of $D$ is the product of the
$k$ complex numbers on the diagonal.  The $j$-th entry of the vector
$V^tD\tilde a^t$ is $f(\tilde w^{(j)}) = 0$; that is,
$V^tD\tilde a^t = \tilde{0}$.
From this it follows that the determinant of $V^tD$
is zero, and that either one of the $k$ monomials $\tilde x^{\tilde p_i}$
yields $0$ when evaluated at $\tilde w$, or some two of the monomials
yield the same value at $\tilde w$.  In either case,
there is a polynomial $g(\tilde x)$, either the difference of
two monomials or itself a monomial, which yields 0 when evaluated at
$\tilde w$; but clearly then, for such a polynomial,
$g(\tilde w^{(j)}) = 0$ for $j\in \nats$.
Then $g\in I$ and its support is contained in that of $f$.
By the minimality assumption, $f = cg$, for some constant $c$.\
\end{proof}
From Lemma~\ref{lmm:mon-binom} it follows easily that each element
of $I_T(\tilde w)$
is a linear combination with complex coefficients of elements of
$I_T(\tilde w)$ which are  monomials and binomials, so that
$I_T(\tilde w)$ is generated as an ideal by finitely many monomials
and binomials.
\begin{lemma}
\label{lmm:p-q}  
If $f = \tilde x^{\tilde p}-\tilde x^{\tilde q}$ is a binomial,
$\tilde w \in (\comps^*)^d$, and
$Z = \{\tilde w \in T^d : f(\tilde w) = 0\}$, then $(f) = I(Z) = I_T(Z)$.
\end{lemma}
\begin{proof}
As already noted, $Z$ is a subgroup of $T^d$. It is closed under the
power operators.  Therefore the ideal $I(Z) = I_T(Z)$ is the intersection
of ideals generated by binomials.  The binomials in $(f)$ are,
up to multiplication by a unit, of the form
$\tilde x^{h\tilde p_1}-\tilde x^{h\tilde p_2}$, all of which lie in $(f)$.
Therefore $I_T(Z) = (f)$.
\end{proof}
\begin{theorem}
\label{thm:rad-power}  
If $I$ is an ideal of $\comps[x_1,\ldots,x_d]^\pm$ that is radical
and power-closed, then $I$ is generated by products of binomials.
\end{theorem}
\begin{proof}
Since $I$ is radical, $I=I(Z(I))$.
Since $I$ is power-closed, $Z(I)$ is closed under the positive power operators.
Then $I$ is the intersection of the ideals $I_T(\tilde w)$ where
$\tilde w\in Z(I)$, and we may choose a finite subset $W\subseteq Z(I)$
for which $I = \bigcap_{\tilde w\in W}I_T(\tilde w)$.

By Lemma~\ref{lmm:mon-binom}
each ideal $I_T(\tilde w)$ is generated by binomials.  Each of these
ideals is radical, so the intersection $I$ coincides with the product, and
it is generated by products of binomials.
\end{proof}

To conclude this episode, we have the following statement.
\begin{theorem}
\label{thm:Galois-connection}  
There is a bijective, order-reversing correspondence between the set of
subsets $Z\subseteq T^d$ that are closed both topologically and under
the mappings ${\tilde{w}}\rightarrow {\tilde{w}}^{(j)}$ for
$j\in \nats$ and the set of
radical power-closed ideals of $\comps[x_1,\ldots,x_d]^\pm$.
The sets $Z$ are the finite unions of topologically closed subgroups
of the torus group.
\end{theorem}
\begin{proof}
Clearly $I_T$, which maps subsets of $T^d$ to ideals, and $Z_T$, which
maps ideals to subsets, are order-reversing.  Also,
$\phi(S) = Z_T(I_T(S)) \supseteq S$ for each set $S$ and
$\tau(I) = I_T((Z_T(I))\supset I$, for each $I$.  Thus
$(I(\cdot), Z(\cdot))$ form a Galois connection; and $\phi$ and
$\tau$ are both closure operators. The lattices of $\phi$-closed sets
and of $\tau$-closed ideals are dually isomorphic.

We need further only to identify the closed sets.

Given an ideal $I$, the set $Z_T(I)$ is a topologically closed subset of
$T^d$ that is also closed under the mappings
$\tilde w\mapsto \tilde w^{(j)}$ for $j\in\nats$. By a
theorem of~\cite{Lawrence-torus},
such sets are finite unions of closed subgroups of $T^d$.
If $f$ is a binomial, then $Z_T(f)$ is a closed subgroup of $T^d$.
As a topological group, it is isomorphic to the product of a finite
cyclic group and  a $(d-1)$-dimensional torus group.  Any such group
can be obtained as the set $Z_T(f)$ for some binomial $f$; and any closed
subgroup of $T^d$ can be obtained as the finite intersection of such subgroups.
If $I= (f_1, \ldots, f_n)$, where the $f_i$'s are binomials, then
$Z_T(I) = \bigcap_{i=1}^m Z_T(f_i)$.  It follows that any closed subgroup of
$T^d$ is of the form $Z_T(I)$, for some ideal $I$.  If
$Z_1, \ldots, Z_m$ are closed subgroups of $T^d$ and $I_1, \ldots, I_m$
are ideals for which $Z_i = Z_T(I_i)$ ($1\le i \le m$),
then $Z_T(I_1\cdots I_m) = \bigcup_{i=1}^m Z_T(I_i) = \bigcup_{i=1}^m Z_i$.
It follows that the sets in the image of $Z_T$ are precisely the finite
unions of closed subgroups of $T^d$.

For any set $S\subseteq T^d$, $I_T(S)$ is a radical power-closed ideal.
Suppose, conversely, that the ideal $I$ of $\comps[x_1,\ldots,x_d]^\pm$
is radical and power-closed. From Theorem~\ref{thm:rad-power}  
it follows that $I$ is generated by products of binomials.
Given a binomial $f$, consider the subgroup $Z_f$ of $T^d$ of points
$\tilde w$ satisfying $f(\tilde w) = 0$.  Then $I_T(Z_f) = (f)$.

It follows that if $S = Z_T(I)$ then $I_T(S)$, being the intersection
of the radical ideals $Z_f$, equals $I$.
\end{proof}

We now look at the radical of an ideal as the intersection of all
the associated prime ideals and what can be said about the radical of a
power-closed ideal in general. Further, we look at whether or not
the two closure operators; power-closure and taking the radical, commute
in general. 

Using Proposition~\ref{prp:Z(f*)}
we can obtain the unique form of $\sqrt{I^{(*)}}$ as
an intersection of the prime ideals associated with $I^{(*)}$
for a principal ideal $I = (f)$. 

For a simple graph on a given set $V$ of vertices call 
a subtree {\em proper} if it contains at least one edge
(equivalently at least two vertices of $V$). Also, for
each subset $A\subseteq V$ we can form a tree $T(A)$
with vertex set $V$ connecting all the vertices of $A$.
For an antichain $\mathcal{A}\subseteq 2^{[d]} = \power(\{1,\ldots,d\})$
let $\mathcal{A}_1$ denote all the 1-element subsets of $\mathcal{A}$
and $\mathcal{A}_{2+}$ all the subsets with two or more elements,
and so we have a partition $\mathcal{A} = \mathcal{A}_1\cup\mathcal{A}_{2+}$.
We then have the following corollary from Proposition~\ref{prp:Z(f*)}.
\begin{corollary}
\label{cor:rad(f*)}
For $f = a_1x_1 + \cdots + a_dx_d\in R$ and $I = (f)$ 
we have $I^{(*)} = (f,f^{(2)},\ldots,f^{(d)})$ and there
is an antichain $\mathcal{A}\subseteq 2^{[d]}$ and proper trees
(actually proper paths)
$\{ T(A) : A\in\mathcal{A}_{2+}\}$ such that 
$\sum_{i\in A}a_i  = 0$ for each $A\in\mathcal{A}_{2+}$ and
\begin{equation}
\label{eqn:rad(f*)}  
\sqrt{I^{(*)}} =
\left(\bigcap_{\{i\}\in\mathcal{A}_1}(x_i)\right) \cap
\left(\bigcap_{A\in\mathcal{A}_{2+}} (x_i-x_j : \{i,j\}\in E(T(A))\right).
\end{equation}
Further, the ideals in the intersection on the right of (\ref{eqn:rad(f*)})
are all prime ideals; the unique primes associated with $I^{(*)}$.
\end{corollary}
By the same token we have a similar statement for a general principal
ideal $I = (f)$ of $R$. By Proposition~\ref{prp:*-gen} we have
the following.
\begin{corollary}
\label{cor:gen-rad(f*)}
For $f = \sum_{i=1}^ka_i{\tilde{x}_i}^{\tilde{p}_i}$ we have $\lambda(f) = k$
and for the principal ideal $I = (f)$ of $R$ we have that 
$I^{(*)} = (f,f^{(2)},\ldots,f^{(k)})$ and there
is an antichain $\mathcal{A}\subseteq 2^{[\lambda(f)]}$ and proper trees
(actually proper paths)
$\{ T(A) : A\in\mathcal{A}_{2+}\}$ such that 
$\sum_{i\in A}a_i  = 0$ for each $A\in\mathcal{A}_{2+}$ and
\begin{equation}
\label{eqn:gen-rad(f*)}  
\sqrt{I^{(*)}} =
\left(\bigcap_{\{i\}\in\mathcal{A}_1}\sqrt{({\tilde{x}_i}^{\tilde{p}_i})}\right)
\cap\left(\bigcap_{A\in\mathcal{A}_{2+}}
\sqrt{({\tilde{x}_i}^{\tilde{p}_i} - {\tilde{x}_j}^{\tilde{p}_j} :
\{i,j\}\in E(T(A)))}\right).
\end{equation}
\end{corollary}
{\sc Note:} The ideals in the intersection on the right of
(\ref{eqn:gen-rad(f*)}) are not necessarily prime ideals as was
the case in (\ref{eqn:rad(f*)}).
\begin{definition}
\label{def:bimonomial}
Call an ideal $B$ of $R$ {\em bimonomial} if it is either
a monomial ideal, generated by monomials $\tilde{x}_i^{\tilde{p}_i}$ of 
$[x_1,\ldots,x_d]$ or it is generated by binomials
$\tilde{x}_i^{\tilde{p}_i} - \tilde{x}_j^{\tilde{p}_j}$.
\end{definition}
As mentioned in the beginning of Section~\ref{sec:defs-obs},
the binomial ideals of $R$ form a natural class of ideals that are
clearly power-closed. As the right hand side of (\ref{eqn:gen-rad(f*)})
involves the radical of bimonomial ideals, it seems therefore 
of interest to discuss the relationship between the radical of
an ideal and its power-closure.

If $f\in\sqrt{I^{(*)}}$ and $i\in\nats$, then $f^N \in I^{(*)}$ for some 
natural $N$ and so $(f^{(i)})^N  = (f^N)^{(i)} \in I^{(*)}$, or
$f^{(i)} \in\sqrt{I^{(*)}}$, and so $\sqrt{I^{(*)}}$ is power-closed.
Therefore we get in general that 
$(\sqrt{I})^{(*)} \subseteq (\sqrt{I^{(*)}})^{(*)} = \sqrt{I^{(*)}}$.
Further, if $I$ is itself power-closed, then 
$(\sqrt{I})^{(*)} = (\sqrt{I^{(*)}})^{(*)} = \sqrt{I^{(*)}}$.
We summarize in the following.
\begin{observation}
\label{obs:rad*}
For any ideal $I$ of $R$ the ideal $\sqrt{I^{(*)}}$ is power-closed
and $(\sqrt{I})^{(*)} \subseteq \sqrt{I^{(*)}}$.
If $I$ is itself power-closed then
$(\sqrt{I})^{(*)} = \sqrt{I^{(*)}} = \sqrt{I}$ and so $\sqrt{I}$
is also power-closed.
\end{observation}
By Observations~\ref{obs:rad*} and~\ref{obs:IJ-p-closed} we have
\begin{corollary}
\label{cor:Z/2}
For any power-closed ideal $I$ of $R$ and for any half integer
$k\in {\nats}/2$ the ideal $I^k$ is also power-closed.
\end{corollary}
It would be nice to have the two closure operators commute,
but that is not the case.
\begin{example}
\label{exa:sqrt-vs-*}  
For the prime ideal $P = (x+y)$ of $\comps[x,y]$ we have
$(\sqrt{P})^{(*)} = P^{(*)} = (x+y,xy)$ and 
$\sqrt{P^{(*)}} = \sqrt{(x+y,xy)} = (x,y)$, so we see that 
$(\sqrt{I})^{(*)} = \sqrt{I^{(*)}}$ does not hold in general.
\end{example}
Recall from Observation~\ref{obs:IJ-p-closed} that if $I_1,\ldots,I_k$ 
are power-closed ideals
then so is their intersection $I_1\cap\cdots\cap I_k$. Since
each prime ideal in the intersection of (\ref{eqn:rad(f*)}) is
clearly power-closed, then so is $\sqrt{I^{(*)}}$ for that 
particular ideal. Also, by Observation~\ref{obs:rad*} we have
that the intersection on the right of (\ref{eqn:gen-rad(f*)})
is clearly power-closed; being the intersection of radicals of
bimonomial ideals.

Since taking the radical of ideals respects intersection; 
$\sqrt{I\cap J} = \sqrt{I}\cap\sqrt{J}$
(See~\cite[Thm~2.7, p.~380]{Hungerford},~\cite[p.~9]{Atiyah-Macdonald}),
Corollary~\ref{cor:gen-rad(f*)} 
states is that for any principal ideal $I = (f)$ of $R$ we have
\begin{equation}
\label{eqn:B_i}
\sqrt{I^{(*)}} = \bigcap_i\sqrt{B_i} = \sqrt{\bigcap_i B_i},
\end{equation}
which by its mere form is power-closed, since each bimonomial ideal
$B_i$ is power-closed.

By the discussion here above it seems natural to pass onto 
$\sqrt{I^{(*)}}$ for the given ideal $I$ for several reasons:
(i) If $I$ is power-closed then so is its radical by 
Observation~\ref{obs:rad*}.
(ii) The zero locus or the affine variety of an ideal is the same as
its radical. 
(iii) From the minimal primary decomposition of an ideal we obtain
a {\em unique} representation of its radical as the intersection
of the prime ideals associated with the ideal. This, perhaps, might 
give insight into the structure of $\sqrt{I^{(*)}}$ and hence $I$ in
terms of more primitive and even simpler entities.

Consider finally a general ideal $I = (f_1,\ldots,f_n)$ of $R$. 
Writing $I_i = (f_i)$ for each $i$ we have $I = \sum_iI_i$ and
hence by Proposition~\ref{prp:closure}
we have $I^{(*)} = \sum_{i=1}^nI_i^{(*)}$.
Since $\sqrt{I+J} = \sqrt{\sqrt{I}+\sqrt{J}}$ and by (\ref{eqn:B_i})
each $\sqrt{I_i^{(*)}} = \sqrt{\bigcap_jB_{i\/j}}$, we then get
for our general ideal $I$ that
\[
\sqrt{I^{(*)}} = \sqrt{\sum_{i=1}^nI_i^{(*)}} 
= \sqrt{\sum_{i=1}^n\sqrt{I_i^{(*)}}} 
= \sqrt{\sum_{i=1}^n\sqrt{\bigcap_j B_{i\/j}}}
= \sqrt{\sum_{i=1}^n\bigcap_j B_{i\/j}},
\]
which also by its mere form is power-closed, since each bimonomial idea
$B_{i\/j}$ is power-closed, and which in particular is consistent with
Observation~\ref{obs:rad*}, but more importantly also reveals a specific
internal bimonomial structure.

\bibliographystyle{amsalpha}
\bibliography{pcix-arXiv}

\end{document}